\documentclass[amsppt,12pt]{article}%
\usepackage{amssymb}
\usepackage{graphicx}
\usepackage{amsmath}
\usepackage{amsmath}
\usepackage{graphicx}
\usepackage{amsfonts}
\usepackage{lscape}
\usepackage{epsfig}%

\providecommand{\U}[1]{\protect\rule{.1in}{.1in}}
\newtheorem{theorem}{Theorem}

\newtheorem{lemma}{Lemma}

\newtheorem{proposition}{Proposition}

\newenvironment{proof}[1][Proof]{\textbf{#1.} }{\ \rule{0.5em}{0.5em}}
\begin{document}

\author{Blanchet, J., Lam, H., and Zwart, B.\\\emph{Columbia University, Boston University, and CWI}}
\title{Efficient Rare-event Simulation for Perpetuities}
\date{}
\maketitle

\begin{abstract}
We consider perpetuities of the form%
\[
D=B_{1}\exp\left(  Y_{1}\right)  +B_{2}\exp\left(  Y_{1}+Y_{2}\right)  +...,
\]
where the $Y_{j}$'s and $B_{j}$'s might be i.i.d. or jointly driven by a
suitable Markov chain. We assume that the $Y_{j}$'s satisfy the so-called
Cram\'{e}r condition with associated root $\theta_{\ast}\in(0,\infty)$ and
that the tails of the $B_{j}$'s are appropriately behaved so that $D$ is
regularly varying with index $\theta_{\ast}$. We illustrate by means of an
example that the natural state-independent importance sampling estimator
obtained by exponentially tilting the $Y_{j}$'s according to $\theta_{\ast}$
fails to provide an efficient estimator (in the sense of appropriately
controlling the relative mean squared error as the tail probability of
interest gets smaller). Then, we construct estimators based on state-dependent
importance sampling that are rigorously shown to be efficient.

\end{abstract}

\section{Introduction\label{SectIntro}}

We consider the problem of developing efficient rare-event simulation
methodology for computing the tail of a perpetuity (also known as infinite
horizon discounted reward). Perpetuities arise in the context of ruin problems
with investments and in the study of financial time series such as ARCH-type
processes (see for example, Embrechts \emph{et al.\ }(1997) and Nyrhinen (2001)).

In the sequel we let $X=\left(  X_{n}:n\geq0\right)  $ be an irreducible
finite state-space Markov chain (see Section 2 for precise definitions). In
addition, let $((\xi_{n},\eta_{n}):n\geq1)$ be a sequence of
i.i.d.\ (independent and identically distributed) two dimensional r.v.'s
(random variables) independent of the process $X$. Given $X_{0}=x_{0}$ and
$D_{0}=d_{0}$ the associated (suitably scaled by a parameter $\Delta>0$)
discounted reward at time $n$ takes the form
\begin{align*}
D_{n}\left(  \Delta\right)   &  =d_{0}+\lambda\left(  X_{1},\eta_{1}\right)
\Delta\exp\left(  S_{1}\right)  +\lambda\left(  X_{2},\eta_{2}\right)
\Delta\exp\left(  S_{2}\right) \\
&  +...+\lambda\left(  X_{n},\eta_{n}\right)  \Delta\exp\left(  S_{n}\right)
\end{align*}
where the accumulated rate process $\left(  S_{k}:k\geq0\right)  $ satisfies
\[
S_{k+1}=S_{k}+\gamma\left(  X_{k+1},\xi_{k+1}\right)  ,
\]
given an initial value $S_{0}=s_{0}$. In order to make the notation compact,
throughout the rest of the paper we shall often omit the explicit dependence
of $\Delta$ in $D_{n}\left(  \Delta\right)  $ and we will simply write $D_{n}%
$. We stress that $\Delta>0$ has been introduced as a scaling parameter which
eventually will be sent to zero. Introducing $\Delta$, as we shall see, will
be helpful in the development of the state-dependent importance sampling
algorithm that we study here.

The functions $(\gamma\left(  x,z\right)  :x\in\mathcal{S},z\in\mathbb{R)}$
and $(\lambda\left(  x,z\right)  :x\in\mathcal{S},z\in\mathbb{R)}$ are
deterministic and represent the discount and reward rates respectively. For
simplicity we shall assume that $\lambda\left(  \cdot\right)  $ is
non-negative. Define%
\begin{align}
\phi_{(s_{0},d_{0},x_{0})}\left(  \Delta\right)   &  \triangleq P\left(
D_{\infty}>1|S_{0}=s_{0},D_{0}=d_{0},X_{0}=x_{0}\right) \nonumber\\
&  =P\left(  T_{\Delta}<\infty|S_{0}=s_{0},D_{0}=d_{0},X_{0}=x_{0}\right)  ,
\label{NoA}%
\end{align}
where $T_{\Delta}=\inf\{n\geq0:D_{n}\left(  \Delta\right)  >1\}$.

Throughout this paper the distributions of $\lambda(x,\eta_{1})$ and
$\gamma(x,\xi_{1})$ are assumed to be known both analytically and via
simulation, as well as the transition probability of the Markov chain $X_{i}$.
Our main focus on this paper is on the efficient estimation via Monte Carlo
simulation of $\phi\left(  \Delta\right)  \triangleq\phi_{(0,0,x_{0})}\left(
\Delta\right)  $ as $\Delta\searrow0$ under the so-called Cram\'{e}r condition
(to be reviewed in Section 2) which in particular implies (see Theorem
\ref{ThmAsympt} below)%
\begin{equation}
\phi\left(  \Delta\right)  =c_{\ast}\Delta^{\theta_{\ast}}(1+o\left(
1\right)  ) \label{1}%
\end{equation}
for a given pair of constants $c_{\ast},\theta_{\ast}\in(0,\infty)$. Note
that
\[
\phi\left(  \Delta\right)  =P\left(  \sum_{k=1}^{\infty}\exp\left(
S_{k}\right)  \lambda\left(  X_{k},\eta_{k}\right)  >\frac{1}{\Delta}\right)
,
\]
so $\Delta$ corresponds to the inverse of the tail parameter of interest.

Although our results will be obtained for $s_{0}=0=d_{0}$, it is convenient to
introduce the slightly more general notation in (\ref{NoA}) to deal with the
analysis of the state-dependent algorithms that we will introduce.

Approximation (\ref{1}) is consistent with well known results in the
literature (e.g. Goldie (1991)) and it implies a polynomial rate of decay to
zero, in $1/\Delta$, for the tail of the distribution of the perpetuity
$\sum_{k=1}^{\infty}\exp\left(  S_{k}\right)  \lambda\left(  X_{k},\eta
_{k}\right)  $. The construction of our efficient Monte Carlo procedures is
based on importance sampling, which is a variance reduction technique popular
in rare-event simulation (see, for instance, Asmussen and Glynn (2008)). It is
important to emphasize that, since our algorithms are based on importance
sampling, they allow to efficiently estimate conditional expectations of
functions of the sample path of $\left\{  D_{n}\right\}  $ given that
$T_{\Delta}<\infty$. The computational complexity analysis of the estimation
of such conditional expectations is relatively straightforward given the
analysis of an importance sampling algorithm based on $\phi\left(
\Delta\right)  $ (see for instance the discussion in Adler, Blanchet and Liu
(2010)). Therefore, as it is customary in rare-event simulation, we
concentrate solely on the algorithmic analysis of a class of estimators for
$\phi\left(  \Delta\right)  $.

Asymptotic approximations related to (\ref{1}) go back to Kesten (1973) who
studied a suitable multidimensional analogue of $D_{\infty}$. In the
one-dimensional setting a key reference is Goldie (1991). Under i.i.d.
assumptions, he gave an expression for the constant $c_{\ast}$ which is only
explicit if $\theta_{\ast}$ is an integer. More recent work on this type of
asymptotics was conducted by Benoite de Saporta (2005), who assumed that the
interest rate process $\{\lambda(X_{n},\eta_{n})\}$ itself forms a finite
state Markov chain, and by Enriquez \emph{et.\ al} (2009), who obtained a
different representation for $c_{\ast}$ if the $X_{i}$ are i.i.d.\ (we will
refer to this important special case as the i.i.d.\ case). Collamore (2009)
studied the case when the sequence $\{\lambda(X_{n},\eta_{n})\}$ is i.i.d.
(not dependent on $X_{n}$) and $\{\gamma(X_{n},\xi_{n})\}$ is modulated by a
Harris recurrent Markov chain $\{X_{n}\}$. Since in our case we need certain
Markovian assumptions that apparently have not been considered in the
literature, at least in the form that we do here, in Section 2, we establish
an asymptotic result of the form (\ref{1}) that fits our framework.

Our algorithms allow to efficiently compute $\phi\left(  \Delta\right)  $ with
arbitrary degree of precision, in contrast to the error implied by asymptotic
approximations such as (\ref{1}). In particular, our algorithms can be used to
efficiently evaluate the constant $c_{\ast}$, whose value is actually of
importance in the statistical theory of ARCH processes (see for example,
Chapter 8 of Embrechts \emph{et al.} (1997)).

The efficiency of our simulation algorithms is tested according to widely
applied criteria in the context of rare event simulation. These efficiency
criteria requires the relative mean squared error of the associated estimator
to be appropriately controlled (see for instance the text of Asmussen and
Glynn (2008)). Let us recall some basic notions on these criteria in
rare-event simulation. An unbiased estimator $Z_{\Delta}$ is said to be
strongly efficient if $E(Z_{\Delta}^{2})=O(\phi\left(  \Delta\right)  ^{2})$.
The estimator is said to be asymptotically optimal if $E\left(  Z_{\Delta}%
^{2}\right)  \leq O(\phi\left(  \Delta\right)  ^{2-\varepsilon})$ for every
$\varepsilon>0$. Jensen's inequality yields $EZ_{\Delta}^{2}\geq\phi\left(
\Delta\right)  ^{2}$, so asymptotic optimality requires the best possible rate
of decay for the second moment, and hence the variance, of the underlying
estimator. Despite being a weaker criterion than strong efficiency, asymptotic
optimality is perhaps the most popular efficiency criterion in the rare-event
simulation literature given its convenience yet sufficiency to capture the
rate of decay.

We shall design both strongly efficient and asymptotically optimal estimators
and explain the advantages and disadvantages behind each of them from an
implementation standpoint. Some of these points of comparison relate to the
infinite horizon nature of $D_{\infty}$. We are interested in studying
unbiased estimators. In addition, besides the efficiency criteria we just
mentioned, at the end we are also interested in being able to estimate the
overall running time of the algorithm and argue that the total computational
cost scales graciously as $\Delta\longrightarrow0$. Our main contributions are
summarized as follows:

\bigskip

1) The development of an asymptotically optimal state-dependent importance
sampling estimator for $\phi\left(  \Delta\right)  $ (see Theorem
\ref{ThmMain1}). The associated estimator is shown to be unbiased and the
expected termination time of the algorithm is of order $O\left(  \log
(1/\Delta)^{p}\right)  $ for some $p<\infty$ (see Proposition
\ref{termination} in Section \ref{SectSDIS_Unbias_RunTime}).

\bigskip

2) An alternative, state-independent, estimator is also constructed which is
strongly efficient, see Proposition \ref{strongefficiency}. The
state-independent estimator, however, often will have to be implemented
incurring in some bias (which can be reduced by increasing the length of a
simulation run).

\bigskip

3) New proof techniques based on Lyapunov inequalities. Although Lyapunov
inequalities have been introduced recently for the analysis of importance
sampling estimators in Blanchet and Glynn (2008), the current setting demands
a different approach for constructing the associated Lyapunov function given
that the analysis of $\phi\left(  \Delta\right)  $ involves both light-tailed
and heavy-tailed features (see the discussion later this section; also see
Proposition \ref{PropLIa} in Section \ref{SectSDIS_Efficiency}).

\bigskip

4) A new class of counter-examples showing that applying a very natural
state-independent importance sampling strategy can in fact lead to infinite
variance (see Section \ref{Subsect_Inf_var}). This contribution adds to
previous work by Glasserman and Kou (1995) and further motivates the
advantages of state-dependent importance sampling.

\bigskip

5) The development of an asymptotic result of the form (\ref{1}) that may be
of independent interest (see Theorem \ref{ThmAsympt}).

\bigskip

As we mentioned earlier, importance sampling is a variance reduction technique
whose appropriate usage leads to an efficient estimator. It consists in
sampling according to a suitable distribution in order to appropriately
increase the frequency of the rare event of interest. The corresponding
estimator is just the indicator function of the event of interest times the
likelihood ratio between the nominal distribution and the sampling
distribution evaluated at the observed outcome. The sampling distribution used
to simulate is said to be the importance sampling distribution or the
change-of-measure. Naturally, in order to design efficient estimators one has
to mimic the behavior of the zero-variance change-of-measure, which coincides
precisely with the conditional distribution of $\{D_{n}\}$ given $D_{\infty
}>1$. Now, assume that $S_{0}=0=D_{0}$. As is known in the literature on ARCH
(it is actually made precise in Enriquez \emph{et.\ al.} (2009)), the event
$T_{\Delta}<\infty$ is typically caused by the event $E_{\Delta}=\{\max
_{k\geq0}S_{k}>\log(1/\Delta)\}$, which is the event that the additive process
$\{S_{k}\}$ hits a large value; see Section 2 for more discussion.
%as we shall discuss in Section 2, the event $T_{\Delta
%}<\infty$ is basically caused by the event that $E_{\Delta}=\{\max
%\{S_{k}:k\geq0\}>\log(1/\Delta)\}$.
In turn, the limiting conditional distribution of the underlying random
variables given $E_{\Delta}$ as $\Delta\downarrow0$ is well understood and
strongly efficient estimators based on importance sampling have been developed
for computing $P(E_{\Delta})$ (see Collamore (2002)). Surprisingly, as we will
show by means of a simple example, directly applying the corresponding
importance sampling estimator which is strongly efficient for $P(E_{\Delta})$
can actually result in infinite variance for any $\Delta>0$ when estimating
$\phi\left(  \Delta\right)  $ (see Section 3.2).

Given the issues raised in the previous paragraph, the development of
efficient simulation estimators for computing $\phi\left(  \Delta\right)  $
calls for techniques that go beyond the direct application of standard
importance sampling estimators. In particular, our techniques are based on
state-dependent importance sampling, which has been substantially studied in
recent years (see for example, Dupuis and Wang (2004 and 2007), Blanchet and
Glynn (2008), and Blanchet, Leder and Glynn (2009)). The work of Dupuis and
Wang provides a criterion, based on a suitable non-linear partial differential
inequality, in order to guarantee asymptotic optimality in light tailed
settings. It is crucial for the development of Dupuis and Wang to have an
exponential rate of decay in the parameter of interest (in our case $1/\Delta
$). Blanchet and Glynn (2008) develop a technique based on Lyapunov
inequalities that provides a criterion that can be used to prove asymptotic
optimality or strong efficiency beyond the exponential decay rate setting.
Such criterion, however, demands the construction of a suitable Lyapunov
function whose nature varies depending on the type of large deviations
environment considered (light vs heavy-tails). The construction of such
Lyapunov functions has been studied in light and heavy-tailed environments
(see for instance, Blanchet, Leder and Glynn (2009) and Blanchet and Glynn (2008)).

The situation that we consider here is novel since it has both light and
heavy-tailed features. On one hand, the large deviations behavior is caused by
the event $E_{\Delta}$, which involves light-tailed phenomena. On the other
hand, the scaling of the probability of interest, namely, $\phi\left(
\Delta\right)  $ is not exponential but polynomial in $1/\Delta$ (i.e. the
tail of the underlying perpetuity is heavy-tailed, in particular, Pareto with
index $\theta_{\ast}$). Consequently, the Lyapunov function required to apply
the techniques in Blanchet and Glynn (2008) includes new features relative to
what has been studied in the literature.

Finally, we mention that while rare event simulation of risk processes has
been considered in the literature (see for instance Asmussen (2000)), such
simulation in the setting of potentially negative interest rates has been
largely unexplored. A related paper is that of Asmussen and Nielsen (1995) in
which deterministic interest rates are considered. A conference proceedings
version of this paper (Blanchet and Zwart (2007), without proofs) considers
the related problem of estimating the tail of a perpetuity with stochastic
discounting, but the discounts are assumed to be i.i.d. and premiums and
claims are deterministic. Finally, we also note a paper by Collamore (2002),
who considered ruin probability of multidimensional random walks that are
modulated by general state space Markov chains.

During the second revision of this paper Collamore et al (2011) proposed an
independent algorithm for the tail distribution of fixed point equations,
which include perpetuities as a particular case. We shall discuss more about
this algorithm in Section 7.

The rest of the paper is organized as follows. In Section 2 we state our
assumptions and review some large deviations results for $\phi\left(
\Delta\right)  $. Section 3 focuses on state-independent sampling. The
state-dependent sampling algorithm is developed in Section 4, and its
efficiency analysis and cost-per-replication are studied in Sections 5 and 6.
In Section 7 we include additional extensions and considerations. Section 8
illustrates our results with a numerical example.

\bigskip

\section{Some Large Deviations for Perpetuities\label{SectLDResults}}

As discussed in the Introduction, we shall assume that the process $\left(
S_{n}:n\geq0\right)  $ is a Markov random walk. As it is customary in the
large deviations analysis of quantities related to these objects, we shall
impose some assumptions in order to guarantee the existence of an asymptotic
logarithmic moment generating function for $S_{k}$.

First we have the following assumption:

\textbf{Assumption 1:} We assume that $X$ is an irreducible Markov chain
taking values in a finite state space $\mathcal{S}$ with transition matrix
$(K\left(  x,y\right)  :x,y\in\mathcal{S})$. Moreover, we further assume that
the $\xi_{k}$'s and $\gamma\left(  \cdot\right)  $ satisfy
\begin{equation}
\sup_{x\in\mathcal{S},\theta\in\mathcal{N}}E\exp\left(  \theta\gamma\left(
x,\xi_{1}\right)  \right)  <\infty, \label{Neigh1}%
\end{equation}
where $\mathcal{N}$ is a neighborhood of the origin.

\bigskip

If Assumption 1 is in force, the Perron-Frobenius theorem for positive and
irreducible matrices guarantees the existence of $\left(  u_{\theta}\left(
x\right)  :x\in\mathcal{S},\theta\in\mathcal{N}\right)  $ and $\exp\left(
\psi\left(  \theta\right)  \right)  $ so that%
\begin{equation}
u_{\theta}\left(  x\right)  =E_{x}[\exp\left(  \theta\gamma\left(  X_{1}%
,\xi_{1}\right)  -\psi\left(  \theta\right)  \right)  u_{\theta}\left(
X_{1}\right)  ]. \label{Eigen1}%
\end{equation}
The function $u_{\theta}\left(  \cdot\right)  $ is strictly positive and
unique up to constant scalings. Indeed, to see how the Perron-Frobenius
theorem is applied, define%
\[
E\exp\left(  \theta\gamma\left(  x,\xi_{1}\right)  \right)  =\exp\left(
\chi\left(  x,\theta\right)  \right)
\]
and note that (\ref{Eigen1}) is equivalent to the eigenvalue problem%
\[
\left(  Q_{\theta}u_{\theta}\right)  \left(  x\right)  =\exp\left(
\psi\left(  \theta\right)  \right)  u_{\theta}\left(  x\right)  ,
\]
where $Q_{\theta}\left(  x,y\right)  =K\left(  x,y\right)  \exp\left(
\chi\left(  y,\theta\right)  \right)  $.

\bigskip

We also impose the following assumption that is often known as Cram\'{e}r's condition.

\bigskip

\textbf{Assumption 2: }Suppose that there exists $\theta_{\ast}>0$ such that
$\psi\left(  \theta_{\ast}\right)  =0$. Moreover, assume that there exists
$\theta>\theta_{\ast}$ such that $\psi\left(  \theta\right)  <\infty$.

\bigskip

In order to better understand the role of Assumption 2, it is useful to note
that under Assumption 1, given $X_{0}=x_{0}$, $\tau\left(  x_{0}\right)
=\inf\{k\geq1:X_{k}=x_{0}\}$ is finite almost surely and $D=\sum_{k=1}%
^{\infty}\exp\left(  S_{k}\right)  \lambda\left(  X_{k},\eta_{k}\right)  $
admits the decomposition
\begin{equation}
D=B+\exp\left(  Y\right)  D^{\prime}, \label{Dec1}%
\end{equation}
where $D^{\prime}$ is identical in distribution to $D$, and%
\begin{equation}
Y=S_{\tau\left(  x_{0}\right)  },\text{ }B=\sum_{j=1}^{\tau\left(
x_{0}\right)  }\lambda\left(  X_{j},\eta_{j}\right)  \exp(S_{j}), \label{Dec2}%
\end{equation}
and $D^{\prime}$ is equal in distribution to $D$ and independent of $\left(
B,S_{\tau\left(  x_{0}\right)  }\right)  $. In other words, $D$ can be
represented as a perpetuity with i.i.d.\ pairs of reward and discount rates.
This decomposition will be invoked repeatedly during the course of our
development. Now, as we shall see in the proof of Theorem \ref{ThmAsympt}
below, it follows that $\theta_{\ast}>0$ appearing in Assumption 2 is the
Cram\'{e}r root associated to $Y$, that is,%
\begin{equation}
E\exp\left(  \theta_{\ast}Y\right)  =1. \label{Etheta*2}%
\end{equation}
Note also that since the moment generating function of $Y$ is convex, and
$\theta_{\ast}>0$, we must have that $EY<0$ and therefore by regenerative
theory we must have that $E\gamma(X_{\infty},\xi_{1})<0$.

\bigskip

An additional final assumption is imposed in our development.

\textbf{Assumption 3: }Assume that $\sup_{x\in\mathcal{S}}E_{x}\lambda\left(
X_{1},\eta_{1}\right)  ^{\alpha}<\infty$ for each $\alpha\in(0,\infty)$.

\bigskip

The following examples are given to illustrate the flexibility of our framework.

\bigskip

\textbf{Example 1} \label{ExARCH} ARCH sequences have been widely used in
exchange rate and log-return models (see for example, Embrechts \emph{et al.}
(1997)). In these models the object of interest, $A_{n}$, are the standard
deviations of the log-return. The simplest case of ARCH sequences is the
ARCH(1) process, which satisfies
\[
A_{n+1}^{2}=(\alpha_{0}+\alpha_{1}A_{n}^{2})Z_{n+1}^{2}.
\]
Typically, the $Z_{n}$'s are i.i.d.\ standard Gaussian random variables, and
$\alpha_{0}>0$ and $\alpha_{1}<1$. The stationary distribution is a
perpetuity. We can directly work with the stationary distribution of this
process or transform the problem into one with constant rewards (equal to
$\alpha_{0}$) by noting that%
\[
T_{n+1}\triangleq\alpha_{0}+\alpha_{1}A_{n+1}^{2}=\alpha_{0}+\alpha_{1}%
(\alpha_{0}+\alpha_{1}A_{n}^{2})Z_{n+1}^{2}=\alpha_{0}+\alpha_{1}T_{n}%
Z_{n+1}^{2}.
\]
We obtain that%
\[
T_{\infty}-\alpha_{0}\overset{D}{=}B_{1}\exp\left(  Y_{1}\right)  +B_{2}%
\exp\left(  Y_{1}+Y_{2}\right)  +...
\]
where $B_{i}=\alpha_{0}$ and $Y_{i}=\log\left(  \alpha_{1}Z_{i}^{2}\right)  $
for $i\geq1$. Assumptions 1 to 3 are in place in this setting.

\bigskip

\textbf{Example 2} \label{ExPERP} A changing economic environment can be
modeled by say, a two-state Markov chain denoting good and bad economic
states. We can then model the discounted return of a long-term investment
under economic uncertainty as a perpetuity with this underlying Markov
modulation. Denoting $X_{i}\in\{\text{good},\text{bad}\},i=1,2,\ldots$ as the
Markov chain, our return can be represented as
\[
D=B_{1}\exp(Y_{1})+B_{2}\exp(Y_{1}+Y_{2})+\cdots
\]
where $B_{i}=\lambda(X_{i},\eta_{i})$ and $Y_{i}=\gamma(X_{i},\xi_{i})$ are
the return and discount rate at time $i$, and $\eta_{i}$ and $\xi_{i}$ are
i.i.d. r.v.'s denoting the individual random fluctuations.

\bigskip

We now state a result that is useful to test the optimality of our algorithms.

\begin{theorem}
\label{ThmAsympt}Under Assumptions 1 to 3,
\[
\phi\left(  \Delta\right)  =c_{\ast}\Delta^{\theta_{\ast}}(1+o\left(
1\right)  )
\]
as $\Delta\searrow0$.
\end{theorem}

\bigskip

The proof of Theorem \ref{ThmAsympt} will be given momentarily, but first let
us discuss the intuition behind the asymptotics described in the theorem. It
is well known that under our assumptions $\exp(\max\{S_{k}:k\geq0\})$ is
regularly varying with index $-\theta_{\ast}<0$ (a brief argument indicating
the main ideas behind this fact is given in Section 3.1). The principle of the
largest jump in heavy-tailed analysis indicates that the large deviations
behavior of $D$ is dictated by a large jump of size $b/\Delta$ arising from
the largest contribution in the sum of the terms defining $D$. Given that the
reward rates are light-tailed, such contribution is likely caused by
$\exp(\max\{S_{k}:k\geq0\})$, as made explicit by Enriquez \emph{et.\ al.}
(2009) in the i.i.d.\ case. Therefore, the large deviations behavior of $D$ is
most likely caused by the same mechanism that causes a large deviations
behavior in $\exp(\max\{S_{k}:k\geq0\})$. This type of intuition will be
useful in developing an efficient importance sampling scheme for estimating
the tail of $D$.

\bigskip

\begin{proof}
[Proof of Theorem \ref{ThmAsympt}]Note that equations (\ref{Dec1}) and
(\ref{Dec2}) allow us to apply Theorem 4.1 from Goldie (1991). In particular,
in order to apply Goldie's results we need to show that
\begin{align}
E_{x_{0}}\exp\left(  \theta_{\ast}S_{\tau\left(  x_{0}\right)  }\right)   &
=1,\label{Rthstar}\\
E_{x_{0}}\exp\left(  \theta S_{\tau\left(  x_{0}\right)  }\right)   &
<\infty, \label{Rbdd}%
\end{align}
for some $\theta>\theta_{\ast}$ and that%
\begin{equation}
E_{x_{0}}B^{\alpha}<\infty\label{Mom1}%
\end{equation}
for some $\alpha>\theta_{\ast}$ (conditions (\ref{Rthstar}), (\ref{Rbdd}) and
(\ref{Mom1}) here correspond to conditions (2.3), (2.4) and (4.2) respectively
in Goldie\ (1991)). First we show (\ref{Rthstar}). Note that equation
(\ref{Eigen1}) implies that the process%
\[
M_{n}^{\theta_{\ast}}=\frac{u_{\theta_{\ast}}\left(  X_{n}\right)  }%
{u_{\theta_{\ast}}\left(  x_{0}\right)  }\exp\left(  \theta_{\ast}%
S_{n}\right)
\]
is a positive martingale. Therefore, we have that%
\begin{align*}
1  &  =E_{x_{0}}M_{n\wedge\tau\left(  x_{0}\right)  }^{\theta_{\ast}}\\
&  =E_{x_{0}}[\exp\left(  \theta_{\ast}S_{\tau\left(  x_{0}\right)  }\right)
I(\tau\left(  x_{0}\right)  \leq n)]\\
&  +E_{x_{0}}\left[  \frac{u_{\theta_{\ast}}\left(  X_{n}\right)  }%
{u_{\theta_{\ast}}\left(  x_{0}\right)  }\exp\left(  \theta_{\ast}%
S_{n}\right)  I(\tau\left(  x_{0}\right)  >n)\right]  .
\end{align*}
By the monotone convergence theorem we have that%
\[
E\left[  \exp\left(  \theta_{\ast}S_{\tau\left(  x_{0}\right)  }\right)
I(\tau\left(  x_{0}\right)  \leq n)\right]  \longrightarrow E\exp\left(
\theta_{\ast}S_{\tau\left(  x_{0}\right)  }\right)
\]
as $n\nearrow\infty$. On the other hand note that (\ref{Eigen1}) implies that
the matrix $\left(  K_{\theta_{\ast}}\left(  x,y\right)  :x,y\in
\mathcal{S}\right)  $ defined via%
\[
K_{\theta_{\ast}}\left(  x,y\right)  =K\left(  x,y\right)  \frac
{u_{\theta_{\ast}}\left(  y\right)  \exp\left(  \chi\left(  y,\theta_{\ast
}\right)  \right)  }{u_{\theta_{\ast}}\left(  x\right)  }%
\]
is an irreducible stochastic matrix. Now let $\left(  \widetilde{K}%
_{\theta_{\ast}}\left(  x,y\right)  :x,y\in\mathcal{S}\text{ }\backslash
\{x_{0}\}\right)  $ be the submatrix of $K_{\theta_{\ast}}\left(
\cdot\right)  $ that is obtained by removing the row and column corresponding
to state $x_{0}$. Observe that
\[
\left(  \widetilde{K}_{\theta_{\ast}}^{n}\mathbf{1}\right)  \left(
x_{0}\right)  =E_{x_{0}}\left[  \frac{u_{\theta_{\ast}}\left(  X_{n}\right)
}{u_{\theta_{\ast}}\left(  x_{0}\right)  }\exp\left(  \theta_{\ast}%
S_{n}\right)  I(\tau\left(  x_{0}\right)  >n)\right]  .
\]
By irreducibility we have that $\widetilde{K}_{\theta_{\ast}}\left(
\cdot\right)  $, therefore%
\[
E_{x_{0}}\left[  \frac{u_{\theta_{\ast}}\left(  X_{n}\right)  }{u_{\theta
_{\ast}}\left(  x_{0}\right)  }\exp\left(  \theta_{\ast}S_{n}\right)
I(\tau\left(  x_{0}\right)  >n)\right]  \longrightarrow0
\]
as $n\nearrow\infty$, obtaining (\ref{Rthstar}). The bound (\ref{Rbdd})
follows easily by noting that%
\begin{align*}
&  E_{x_{0}}\left[  \exp\left(  \theta S_{\tau\left(  x_{0}\right)  }\right)
\right] \\
&  =\sum_{k=1}^{\infty}E_{x_{0}}\left[  \exp\left(  \theta S_{k}\right)
;\tau\left(  x_{0}\right)  >k-1,X_{k}=x_{0}\right] \\
&  =\sum_{k=1}^{\infty}E_{x_{0}}[\exp\left(  \theta S_{k-1}\right)  K\left(
X_{k-1},x_{0}\right)  \exp\left(  \chi\left(  \theta,x_{0}\right)  \right)
;\tau\left(  x_{0}\right)  >k-1].
\end{align*}
So, if we define, for $x\neq x_{0}$%
\[
v_{\theta}\left(  x\right)  =\exp\left(  \chi\left(  \theta,x_{0}\right)
\right)  K\left(  x,x_{0}\right)  \frac{u_{\theta}\left(  x_{0}\right)
}{u_{\theta}\left(  x\right)  },
\]
and
\[
R_{\theta}\left(  x,y\right)  =\exp\left(  \chi\left(  \theta,y\right)
\right)  K\left(  x,y\right)  \frac{u_{\theta}\left(  y\right)  }{u_{\theta
}\left(  x\right)  }%
\]
for $x,y\neq x_{0}$ we see that%
\[
E_{x_{0}}\left[  \exp\left(  \theta S_{\tau\left(  x_{0}\right)  }\right)
\right]  =\sum_{k=1}^{\infty}\left(  R_{\theta}^{k}v_{\theta}\right)  \left(
x_{0}\right)  .
\]
Note that $R_{\theta_{\ast}}=\widetilde{K}_{\theta_{\ast}}$ is strictly
substochastic. So, by continuity there exists $\theta>\theta_{\ast}$ for which
$R_{\theta}$ has a spectral radius which is strictly less than one and
therefore (\ref{Rbdd}) holds. Finally, we establish (\ref{Mom1}). Observe that%
\[
B\leq\tau\left(  x_{0}\right)  \max_{1\leq k\leq\tau\left(  x_{0}\right)
}\lambda\left(  X_{k},\eta_{k}\right)  \exp\left(  \max\{S_{k}:1\leq k\leq
\tau\left(  x_{0}\right)  \}\right)  .
\]
Therefore, for $1/p+1/q+1/r=1$ and $p,q,r>1$ we have that%
\begin{align*}
EB^{\alpha}  &  \leq E\tau\left(  x_{0}\right)  ^{\alpha}\max_{1\leq k\leq
\tau\left(  x_{0}\right)  }\lambda\left(  X_{k},\eta_{k}\right)  ^{\alpha}%
\exp\left(  \max\{\alpha S_{k}:1\leq k\leq\tau\left(  x_{0}\right)  \}\right)
\\
&  \leq(E_{x_{0}}\tau\left(  x_{0}\right)  ^{p\alpha})^{1/p}\times(E_{x_{0}%
}\max_{1\leq k\leq\tau\left(  x_{0}\right)  }\lambda\left(  X_{k},\eta
_{k}\right)  ^{q\alpha})^{1/q}\\
&  \times(E_{x_{0}}\exp\left(  \max\{r\alpha S_{k}:1\leq k\leq\tau\left(
x_{0}\right)  \}\right)  )^{1/r}.
\end{align*}
Since $E_{x_{0}}\tau\left(  x_{0}\right)  ^{p\alpha}+E_{x_{0}}\max_{1\leq
k\leq\tau\left(  x_{0}\right)  }\lambda\left(  X_{k},\eta_{k}\right)
^{q\alpha}<\infty$ for all $p,q\in(0,\infty)$ it suffices to show that%
\begin{equation}
E_{x_{0}}\exp\left(  \max\{\theta S_{k}:1\leq k\leq\tau\left(  x_{0}\right)
\}\right)  <\infty\label{Rbdd2}%
\end{equation}
for some $\theta>\theta_{\ast}$. In order to do this define $T\left(
b\right)  =\inf\{k\geq0:S_{k}>b\}$ and note that%
\[
P_{x_{0}}(\max\{S_{k}:1\leq k\leq\tau\left(  x_{0}\right)  \}>b)\leq P_{x_{0}%
}(T\left(  b\right)  \leq\tau\left(  x_{0}\right)  )
\]
Bound (\ref{Rbdd}) implies that $S_{\tau\left(  x_{0}\right)  }$ decays at an
exponential rate that is strictly larger than $\theta_{\ast}$. To analyze
$P_{x_{0}}(T\left(  b\right)  \leq\tau\left(  x_{0}\right)  )$, let $P_{x_{0}%
}^{\theta_{\ast}}\left(  \cdot\right)  $ (respectively $E_{x_{0}}%
^{\theta_{\ast}}\left(  \cdot\right)  $) be the probability measure (resp. the
expectation operator) associated to the change-of-measure induced by the
martingale $(M_{n}^{\theta_{\ast}}:n\geq0)$ introduced earlier. Note that%
\begin{align*}
P_{x_{0}}(T\left(  b\right)   &  \leq\tau\left(  x_{0}\right)  )=E_{x_{0}%
}^{\theta_{\ast}}\left(  \exp\left(  -\theta_{\ast}S_{T\left(  b\right)
}\right)  \frac{u_{\theta_{\ast}}\left(  x_{0}\right)  }{u_{\theta_{\ast}%
}\left(  X_{T\left(  b\right)  }\right)  }I\left(  T\left(  b\right)  \leq
\tau\left(  x_{0}\right)  \right)  \right) \\
&  \leq c\exp\left(  -\theta_{\ast}b\right)  P_{x_{0}}^{\theta_{\ast}}\left(
\tau\left(  x_{0}\right)  \geq T\left(  b\right)  \right)  .
\end{align*}
The probability $P_{x_{0}}^{\theta_{\ast}}\left(  \tau\left(  x_{0}\right)
>T\left(  b\right)  \right)  $ decays exponentially fast as $b\nearrow\infty$
by using a standard large deviations argument on $T\left(  b\right)  /b$ and
because $\tau\left(  x_{0}\right)  $ has exponentially decaying tails.
Therefore we obtain (\ref{Rbdd2}), which in turn yields the conclusion of the theorem.
\end{proof}

\bigskip

\section{State-Independent Importance
Sampling\label{SectIndependentIS_Example}}

In order to design efficient estimators we will apply importance sampling. It
is well known (see for example, Liu (2001)) that the zero-variance importance
sampling distribution is dictated by the conditional distribution of the
process (in this case the triplet $\left\{  \left(  S_{n},D_{n},X_{n}\right)
\right\}  $), given the occurrence of the rare event in question. As we
discussed in the previous section, the occurrence of the event $T_{\Delta
}<\infty$ is basically driven by the tail behavior of $\max\{S_{n}:n\geq0\}$.
In turn, the large deviations behavior of Markov random walks (such as $S$) is
well understood from a simulation standpoint and, under our assumptions, there
is a natural state-independent change-of-measure that can be shown to be
efficient for estimating the tail of $\max\{S_{n}:n\geq0\}$.

The present section is organized as follows. First, we shall explain the
change-of-measure that is efficient for estimating the tail of $\max
\{S_{n}:n\geq0\}$ because it will serve as the basis for our change-of-measure
in the setting of perpetuities (but some modifications are crucial to
guarantee good performance). After that, we shall show by means of an example
that this type of importance sampling algorithm can lead to estimators that
have infinite variance. We then close this section with a modified
state-independent importance sampling algorithm that is strongly efficient but biased.

\subsection{The standard approach\label{SecStandardApproach}}

As we indicated in the proof of Theorem \ref{ThmAsympt}, given $\theta$ and
$X_{0}=x_{0}$ equation (\ref{Eigen1}) indicates that the process%
\[
M_{n}^{\theta}=\exp(\theta S_{n}-n\psi\left(  \theta\right)  )\frac{u_{\theta
}\left(  X_{n}\right)  }{u_{\theta}\left(  x_{0}\right)  },\text{ \ }n\geq0
\]
is a positive martingale as long as $\psi\left(  \theta\right)  <\infty$ and
therefore it generates a change-of-measure. The probability measure in
path-space induced by this martingale is denoted by $P_{x_{0}}^{\theta}%
(\cdot)$ and in order to simulate the process $((D_{n},S_{n},X_{n}):n\geq0)$
according to $P_{x_{0}}^{\theta}(\cdot)$ one proceeds as follows.

\begin{enumerate}
\item Generate $X_{1}$ according to the transition matrix%
\[
K_{\theta}\left(  x,y\right)  =K\left(  x,y\right)  \exp\left(  \chi\left(
y,\theta\right)  -\psi\left(  \theta\right)  \right)  u_{\theta}\left(
y\right)  /u_{\theta}\left(  x\right)  ,
\]
which is guaranteed to be a Markov transition matrix by the definition of
$u_{\theta}\left(  \cdot\right)  $ and $\chi\left(  \cdot,\theta\right)  $.

\item Given $X_{1}=y$, sample $\gamma\left(  y,\xi_{1}\right)  $ according to
exponential tilting given by%
\[
P^{\theta}\left(  \gamma\left(  y,\xi_{1}\right)  \in dz\right)  =\exp\left(
\theta z-\chi\left(  y,\theta\right)  \right)  P\left(  \gamma\left(
y,\xi_{1}\right)  \in dz\right)  .
\]

\item Simulate $\lambda\left(  y,\eta_{1}\right)  $ according to the nominal
conditional distribution of $\eta_{1}$ given that $X_{1}=y$ and $\gamma\left(
y,\xi_{1}\right)  $.
\end{enumerate}

The previous rules allow to obtain $(D_{1},S_{1},X_{1})$ given $(D_{0}%
,S_{0},X_{0})$. Subsequent steps are performed in a completely analogous way.

\bigskip

Note that if one selects $\theta=\theta_{\ast}$ then we have that $\psi\left(
\theta_{\ast}\right)  =0$ and also $S_{n}/n\longrightarrow\psi^{\prime}\left(
\theta_{\ast}\right)  >0$ a.s. with respect to $P_{x_{0}}^{\theta_{\ast}%
}\left(  \cdot\right)  $. If $\sigma\left(  y\right)  =\inf\{n\geq0:S_{n}>y\}$
then%
\begin{align*}
P_{x_{0}}\left(  \max\{S_{n}:n\geq0\}>y\right)   &  =P_{x_{0}}(\sigma\left(
y\right)  <\infty)\\
&  =E_{x_{0}}^{\theta_{\ast}}[1/M_{\sigma\left(  y\right)  }^{\theta_{\ast}%
};\sigma(y)<\infty]\\
&  =E_{x_{0}}^{\theta_{\ast}}\left(  \exp(-\theta_{\ast}S_{\sigma\left(
y\right)  })\frac{u_{\theta_{\ast}}\left(  x_{0}\right)  }{u_{\theta_{\ast}%
}\left(  X_{\sigma\left(  y\right)  }\right)  }\right)  .
\end{align*}
In the last equality we have used that $\sigma(y)<\infty$ a.s. with respect to
$P_{x}^{\theta_{\ast}}\left(  \cdot\right)  $. The importance sampling
estimator%
\[
Z_{y}^{\prime}=\exp(-\theta_{\ast}S_{\sigma\left(  y\right)  })\frac
{u_{\theta_{\ast}}\left(  x_{0}\right)  }{u_{\theta_{\ast}}\left(
X_{\sigma\left(  y\right)  }\right)  },
\]
obtained by sampling according to $P_{x_{0}}^{\theta_{\ast}}\left(
\cdot\right)  $ is unbiased and its second moment satisfies%
\begin{align*}
E_{x_{0}}^{\theta_{\ast}}\left(  Z_{y}^{\prime2}\right)   &  =E_{x_{0}%
}^{\theta_{\ast}}\exp(-2\theta_{\ast}S_{\sigma\left(  y\right)  }%
)\frac{u_{\theta_{\ast}}^{2}\left(  x_{0}\right)  }{u_{\theta_{\ast}}%
^{2}\left(  X_{\sigma\left(  y\right)  }\right)  }\\
&  \leq c\exp\left(  -2\theta_{\ast}y\right)  ,
\end{align*}
for some constant $c\in(0,\infty)$. Since, under our assumptions (see for
example, Asmussen (2003), Theorems 5.2 and 5.3 on Page 365) $P_{x_{0}}%
(\sigma\left(  y\right)  <\infty)=\gamma_{\ast}\exp\left(  -\theta_{\ast
}y\right)  (1+o(1))$ as $y\nearrow\infty$ for a suitable $\gamma_{\ast}%
\in(0,\infty)$ we obtain that $Z_{y}^{\prime}$ is strongly efficient for
estimating $P_{x_{0}}(\sigma\left(  y\right)  <\infty)$ as $y\nearrow\infty$.

\subsection{Standard importance sampling can lead to infinite
variance\label{Subsect_Inf_var}}

A natural approach would be to apply directly the previous change-of-measure
for estimating $T_{\Delta}<\infty$. Nevertheless, we will show by means of a
simple continuous-time example that can be fit in the discrete time setting
(through simple embedding) that such approach does not guarantee efficiency.
In fact, the estimator might even have infinite variance.

\bigskip

\textbf{Example 3}: Let $\left(  X\left(  t\right)  :t\geq0\right)  $ be
Brownian motion with negative drift $-\mu$ and unit variance. We are concerned
with $\phi\left(  \Delta\right)  =P(\int_{0}^{\infty}\exp\left(  X\left(
t\right)  \right)  dt>1/\Delta)$. In particular, here we have $\psi\left(
\theta\right)  =t^{-1}\log E\exp\left(  \theta X\left(  t\right)  \right)
=-\mu\theta+\theta^{2}/2$ (in the discrete time setting $\psi\left(
\theta\right)  $ will be $-\mu\theta+\theta^{2}/2$ multiplied by the
discretization time scale). In order to analyze the second moment of the
natural estimator described previously we need to evaluate $\theta_{\ast}$ so
that $\psi\left(  \theta_{\ast}\right)  =0$. This yields $\theta_{\ast}=2\mu$
and the resulting importance sampling algorithm proceeds by simulating
$X\left(  \cdot\right)  $ according to a Brownian motion with
\textit{positive} drift $\mu$ and unit variance up to time $T_{\Delta}%
=\inf\{t\geq0:\int_{0}^{t}\exp\left(  X\left(  s\right)  \right)
ds\geq1/\Delta\}$ and returning the estimator $Z_{\Delta}=\exp\left(
-\theta_{\ast}X\left(  T_{\Delta}\right)  \right)  $. The second moment of the
estimator is then given by $E^{\theta_{\ast}}[Z_{\Delta}^{2};T_{\Delta}%
<\infty]=E[Z_{\Delta};T_{\Delta}<\infty]$, by a change of measure back to the
original one (here we have used $E^{\theta_{\ast}}\left(  \cdot\right)  $ to
denote probability measure under which $X\left(  \cdot\right)  $ follows a
Brownian motion with drift $\mu$ and unit variance). We will show that if
$\theta_{\ast}\geq1$ then $E[\left.  \exp\left(  -\theta_{\ast}X\left(
T_{\Delta}\right)  \right)  \right\vert T_{\Delta}<\infty]=\infty$. By the
definition of $\theta_{\ast}$ we will conclude that if $\mu\geq1/2$ then
\[
E\left[  Z_{\Delta};T_{\Delta}<\infty\right]  =E[\left.  \exp\left(
-\theta_{\ast}X\left(  T_{\Delta}\right)  \right)  \right\vert T_{\Delta
}<\infty]P\left(  T_{\Delta}<\infty\right)  =\infty
\]
which implies infinite variance of the estimator. In order to prove this we
will use a result of Pollack and Siegmund (1985) (see also Dufresne (1990))
which yields that $D=\int_{0}^{\infty}\exp\left(  X\left(  t\right)  \right)
dt$ is equal in distribution to $1/Z$, where $Z$ is distributed gamma with a
density of the form%
\[
f_{Z}\left(  z\right)  =\frac{\exp\left(  -\lambda z\right)  \lambda
^{\theta_{\ast}}z^{\theta_{\ast}-1}}{\Gamma\left(  \theta_{\ast}\right)  },
\]
for some $\lambda>0$. In particular, a transformation of variable gives the
density of $D$ as%
\[
f_{D}(y)=\frac{\exp(-\lambda/y)\lambda^{\theta_{\ast}}}{\Gamma(\theta_{\ast
})y^{\theta_{\ast}+1}}%
\]
and hence%
\[
P\left(  D>1/\Delta\right)  =\int_{1/\Delta}^{\infty}\frac{\exp(-\lambda
/y)\lambda^{\theta_{\ast}}}{\Gamma(\theta_{\ast})y^{\theta_{\ast}+1}}%
dy=\frac{\exp(-\lambda\Delta)\lambda^{\theta_{\ast}}\Delta^{\theta_{\ast}}%
}{\theta_{\ast}\Gamma(\theta_{\ast})}+\int_{1/\Delta}^{\infty}\frac
{\exp(-\lambda/y)\lambda^{\theta_{\ast}+1}}{\theta_{\ast}\Gamma(\theta_{\ast
})y^{\theta_{\ast}+2}}dy
\]
where the second inequality follows from integration by parts. Note that
\[
\int_{1/\Delta}^{\infty}\frac{\exp(-\lambda/y)\lambda^{\theta_{\ast}+1}%
}{\theta_{\ast}\Gamma(\theta_{\ast})y^{\theta_{\ast}+2}}dy\leq\frac
{\lambda^{\theta_{\ast}+1}}{\theta_{\ast}\Gamma(\theta_{\ast})}\int_{1/\Delta
}^{\infty}\frac{1}{y^{\theta_{\ast}+2}}dy=\frac{\lambda^{\theta_{\ast}%
+1}\Delta^{\theta_{\ast}+1}}{\theta_{\ast}(\theta_{\ast}+1)\Gamma(\theta
_{\ast})}=O(\Delta^{\theta_{\ast}+1})
\]
as $\Delta\searrow0$. Hence $P\left(  D>1/\Delta\right)  \sim c\Delta
^{\theta_{\ast}}$ where $c=\exp(-\lambda\Delta)\lambda^{\theta_{\ast}}%
/(\theta_{\ast}\Gamma(\theta_{\ast}))$. Now, let $W$ be a random variable
equal in law to $\Delta\lbrack D-1/\Delta]$ given that $D>1/\Delta$ . Note
that conditional on $T_{\Delta}<\infty$, we have%
\begin{align*}
D  &  =\int_{0}^{\infty}\exp(X(t))dt\\
&  =\int_{0}^{T_{\Delta}}\exp(X(t))dt+\int_{T_{\Delta}}^{\infty}\exp(X(t))dt\\
&  =\frac{1}{\Delta}+\exp(X(T_{\Delta}))\int_{T_{\Delta}}^{\infty}%
\exp(X(t)-X(T_{\Delta}))dt\\
&  =\frac{1}{\Delta}+\exp(X(T_{\Delta}))D^{\prime}%
\end{align*}
where $D^{\prime}$ has the same distribution as $1/Z$ and is independent of
$\exp\left(  X\left(  T_{\Delta}\right)  \right)  $ given that $T_{\Delta
}<\infty$. We have used the strong Markov property and stationary increment
property of Brownian motion in the fourth equality. Hence the random variable
$X\left(  T_{\Delta}\right)  $ satisfies the equality in distribution%
\[
W=_{d}\Delta\exp\left(  X\left(  T_{\Delta}\right)  \right)  D^{\prime},
\]
Now, it is clear that $E[\left(  D^{\prime}\right)  ^{-\theta_{\ast}%
}]=EZ^{\theta_{\ast}}<\infty$, so it suffices to show that if $\theta_{\ast
}\geq1$ then $E\left(  W^{-\theta_{\ast}}\right)  =\infty$. Using the
definition of $W$, transformation of variable gives%
\[
f_{W}\left(  w\right)  =\frac{\Delta^{-1}\lambda^{\theta_{\ast}}}%
{\Gamma\left(  \theta^{\ast}\right)  P\left(  D>1/\Delta\right)  }\exp\left(
-\lambda/(w/\Delta+1/\Delta)\right)  (w/\Delta+1/\Delta)^{-\left(
\theta_{\ast}+1\right)  }.
\]
Therefore, we have that there exists a constant $c_{0}\in(0,\infty)$ such that%
\begin{align*}
E\left(  W^{-\theta_{\ast}}\right)   &  =\int_{0}^{\infty}f_{W}\left(
w\right)  w^{-\theta_{\ast}}dw\\
&  \geq c_{0}\int_{0}^{\infty}\Delta^{-1}\exp\left(  -\lambda/(w/\Delta
+1/\Delta)\right)  (w/\Delta+1/\Delta)^{-\left(  \theta_{\ast}+1\right)
}\left(  \Delta w\right)  ^{-\theta_{\ast}}dw\\
&  \geq c_{0}\int_{0}^{1}\exp\left(  -\lambda\Delta\right)  2^{-\left(
\theta_{\ast}+1\right)  }w^{-\theta_{\ast}}dw=\infty.
\end{align*}

\bigskip

The problem behind the natural importance sampling estimator is that one would
like the difference $[S_{T_{\Delta}}-\log(1/\Delta)]$ to stay positive, but
unfortunately, this cannot be guaranteed and in fact, this difference will
likely be negative. The idea that we shall develop in the next subsection is
to apply importance sampling just long enough to induce the rare event.

\subsection{A modified algorithm\label{modifiedalgorithm}}

We select $\theta=\theta_{\ast}$ and simulate the process according to the
procedure described in Steps 1 to 3 explained in Section 3.1 up to time
\[
T_{\Delta/a}=\inf\{n\geq0:D_{n}>a\},
\]
for some $a\in\left(  0,1\right)  $. Subsequent steps of the process $\left\{
(S_{k},D_{k},X_{k})\right\}  $, for $k>T_{\Delta/a}$ are simulated under the
nominal (original) dynamics up until $T_{\Delta}$. The resulting estimator
takes the form%
\begin{equation}
Z_{1,\Delta}=\exp\left(  -\theta_{\ast}S_{T_{\Delta/a}}\right)  \frac
{u_{\theta_{\ast}}\left(  x_{0}\right)  }{u_{\theta_{\ast}}\left(
X_{T_{\Delta/a}}\right)  }I(D_{\infty}>1). \label{Estimator1}%
\end{equation}

We will discuss the problem of implementing this estimator in a moment, in
particular the problem of sampling $I(D_{\infty}>1)$ in finite time. First we
examine its efficiency properties. We assume for simplicity that the rewards
are bounded, we will discuss how to relax this assumption right after the
proof of this result.

\bigskip

\begin{theorem}
\label{strongefficiency}In addition to Assumptions 1 and 2, suppose that there
is deterministic constant $m\in\left(  0,\infty\right)  $ such that
$\lambda\left(  x,\eta\right)  <$ $m$. Then, $Z_{1,\Delta}$ is a strongly
efficient estimator of $\phi\left(  \Delta\right)  $.
\end{theorem}

\begin{proof}
It is clear that $Z_{1,\Delta}$ is unbiased. Now, define
\[
U_{n}=\exp\left(  -S_{n}\right)  \{D_{n}-a\}/\Delta
\]
for $n\geq1$. The process $\left\{  U_{n}\right\}  $ will be helpful to study
the overshoot at time $T_{\Delta/a}$. Note that%
\begin{equation}
U_{n+1}=\lambda\left(  X_{n+1},\eta_{n+1}\right)  +\exp\left(  -\gamma\left(
X_{n+1},\xi_{n+1}\right)  \right)  U_{n}, \label{DefU}%
\end{equation}
and also that we can write $T_{\Delta/a}=\inf\{n\geq0:U_{n}>0\}$.

It is important to observe that%
\begin{equation}
D_{\infty}=a+\exp\left(  S_{T_{\Delta/a}}\right)  \Delta U_{T_{\Delta/a}}%
+\exp\left(  S_{T_{\Delta/a}}\right)  D_{\infty}^{\prime}, \label{Dec3}%
\end{equation}
where $D_{\infty}^{\prime}$ is conditionally independent of $S_{T_{\Delta/a}}%
$, $U_{T_{\Delta/a}}$ given $X_{T_{\Delta/a}}$. In addition, $D_{\infty
}^{\prime}$ is obtained from the original / nominal distribution.
Decomposition (\ref{Dec3}) implies that%
\[
I(D_{\infty}>1)\leq I(\exp\left(  S_{T_{\Delta/a}}\right)  D_{\infty}^{\prime
}>(1-a)/2)+I(\exp\left(  S_{T_{\Delta/a}}\right)  \Delta U_{T_{\Delta/a}%
}>(1-a)/2),
\]
and therefore, by conditioning on $(S_{n},D_{n},X_{n})$ for $n\leq
T_{\Delta/a}$ we obtain that the second moment of $Z_{1,\Delta}$ is bounded by%
\begin{align}
&  E_{(0,0,x_{0})}^{\theta_{\ast}}[\exp\left(  -2\theta_{\ast}S_{T_{\Delta/a}%
}\right)  \frac{u_{\theta_{\ast}}\left(  x_{0}\right)  }{u_{\theta_{\ast}%
}\left(  X_{T_{\Delta/a}}\right)  }\phi_{(0,0,X_{T_{\Delta/a}})}\left(
\exp\left(  S_{T_{\Delta/a}}\right)  2\Delta/(1-a)\right)  ]\label{BD1}\\
&  \left.  +E_{(0,0,x_{0})}^{\theta_{\ast}}[\exp\left(  -2\theta_{\ast
}S_{T_{\Delta/a}}\right)  \frac{u_{\theta_{\ast}}\left(  x_{0}\right)
}{u_{\theta_{\ast}}\left(  X_{T_{\Delta/a}}\right)  }I(\exp\left(
S_{T_{\Delta/a}}\right)  \Delta U_{T_{\Delta/a}}>(1-a)/2)]\right.  .
\label{BD2}%
\end{align}
We will denote by $I_{1}$ the term in (\ref{BD1}) and by $I_{2}$ the term in
(\ref{BD2}).\ It suffices to show that both $I_{1}$ and $I_{2}$ are of order
$O\left(  \Delta^{2\theta_{\ast}}\right)  $.

Theorem \ref{ThmAsympt} guarantees the existence of a constant $c_{1}%
\in(1,\infty)$ so that%
\[
\phi_{(0,0,x_{0})}\left(  \Delta\right)  \leq c_{1}\exp\left(  \theta_{\ast
}S_{T_{\Delta/a}}\right)  \Delta^{\theta_{\ast}}/(1-a)^{\theta_{\ast}}.
\]
Using this bound inside (\ref{BD1}) we obtain that%
\begin{equation}
I_{1}\leq m_{1}\frac{\Delta^{\theta_{\ast}}}{(1-a)^{\theta_{\ast}}%
}E_{(0,0,x_{0})}^{\theta_{\ast}}[\exp\left(  -2\theta_{\ast}S_{T_{\Delta/a}%
}\right)  \frac{u_{\theta_{\ast}}\left(  x_{0}\right)  }{u_{\theta_{\ast}%
}\left(  X_{T_{\Delta/a}}\right)  }] \label{Bd2b}%
\end{equation}
for some constant $m_{1}>0$ and thus, since
\begin{equation}
E_{(0,0,x_{0})}^{\theta_{\ast}}[\exp\left(  -\theta_{\ast}S_{T_{\Delta/a}%
}\right)  \frac{u_{\theta_{\ast}}\left(  x_{0}\right)  }{u_{\theta_{\ast}%
}\left(  X_{T_{\Delta/a}}\right)  }]=\phi\left(  \Delta\right)  =O\left(
\Delta^{\theta_{\ast}}\right)  \label{IDI1}%
\end{equation}
we conclude that $I_{1}=O\left(  \Delta^{2\theta_{\ast}}\right)  $.

We now study the term $I_{2}$. Just as in the proof of Markov's inequality,
note that for any $\beta>0$%
\begin{equation}
I_{2}\leq\Delta^{\beta}\left(  \frac{2}{1-\alpha}\right)  ^{\beta
}E_{(0,0,x_{0})}^{\theta_{\ast}}[\exp\left(  -2\theta_{\ast}S_{T_{\Delta/a}%
}\right)  \exp\left(  \beta S_{T_{\Delta/a}}\right)  \frac{u_{\theta_{\ast}%
}\left(  x_{0}\right)  }{u_{\theta_{\ast}}\left(  X_{T_{\Delta/a}}\right)
}U_{T_{\Delta/a}}^{\beta}]. \label{BDI2a}%
\end{equation}
We could pick, for instance, $\beta=3\theta_{\ast}/2$ and use the fact that
$u_{\theta_{\ast}}\left(  X_{T_{\Delta/a}}\right)  \geq\delta$ for some
$\delta>0$ to obtain that
\begin{align}
&  E_{(0,0,x_{0})}^{\theta_{\ast}}[\exp\left(  -2\theta_{\ast}S_{T_{\Delta/a}%
}\right)  \exp\left(  \beta S_{T_{\Delta/a}}\right)  \frac{u_{\theta_{\ast}%
}\left(  x_{0}\right)  }{u_{\theta_{\ast}}\left(  X_{T_{\Delta/a}}\right)
}U_{T_{\Delta/a}}^{\beta}]\label{BDI2b}\\
&  \leq\frac{u_{\theta_{\ast}}\left(  x_{0}\right)  }{\delta}\left(
E_{(0,0,x_{0})}^{\theta_{\ast}}[\exp\left(  -\theta_{\ast}S_{T_{\Delta/a}%
}\right)  ]\right)  ^{1/2}(E_{(0,0,x_{0})}^{\theta_{\ast}}[U_{T_{\Delta/a}%
}^{2\beta}])^{1/2}.\nonumber
\end{align}
If we are able to show that
\begin{equation}
E_{(0,0,x_{0})}^{\theta_{\ast}}(U_{T_{\Delta/a}}^{2\beta})=O\left(  1\right)
\label{BDI2c}%
\end{equation}
as $\Delta\searrow0$, then we an conclude, owing to (\ref{IDI1}), that the
right hand side of (\ref{BDI2b}) is of order $O\left(  \Delta^{\theta_{\ast
}/2}\right)  $. Thus, combining this bound on (\ref{BDI2b}), together with
(\ref{BDI2a}) we would conclude that $I_{2}=O\left(  \Delta^{2\theta_{\ast}%
}\right)  $ as required. It suffices then to verify (\ref{BDI2c}), however,
this is immediate since under our current assumptions we clearly have that
$U_{T_{\Delta/a}}\leq\lambda\left(  X_{T_{\Delta/a}},\eta_{T_{\Delta/a}%
}\right)  \leq m$.
\end{proof}

\bigskip

We shall comment on two important issues behind this result. First, we have
assumed that the rewards are bounded in order to simplify our analysis. Note
that the only place that used this assumption is in establishing
(\ref{BDI2c}). It is possible to estimate the expectation in (\ref{BDI2c})
only under Assumption 3 using a Lyapunov bound similar to the one that we will
discuss in Lemma \ref{LemLI}.

Second, the estimator $Z_{1,\Delta}$ is unbiased only if we can generate
$D_{\infty}$ in a finite time. Generating unbiased samples from $D_{\infty}$
under our current assumptions is not straightforward (see for example Diaconis
and Freedman (1999) on issues related to steady-state distributions for
iterated random functions, and Blanchet and Sigman (2011) for algorithms that
can be used to sample $D_{\infty}$ under assumptions close to the ones that we
impose here). Alternatively, one might recognize that $D_{\infty}$ is the
steady-state distribution of a suitably defined Markov chain. In the presence
of enough regeneration structure, one can replace the indicator in
(\ref{Estimator1}) by an estimator for the tail of $D_{\infty}$ based on the
corresponding regenerative ratio representation. Note that this replacement
would involve a routine simulation problem as there is no need to estimate any
rare event. However, once again after using a regenerative-ratio based
estimator one introduces bias.

We shall not pursue more discussion on any of the two issues raised given that
the class of estimators that we shall discuss in the next section are not only
unbiased but are also asymptotically optimal as $\Delta\longrightarrow0$ and
can be rigorously shown to have a running time that grows at most
logarithmically in $1/\Delta$.

\section{State-Dependent Importance Sampling\label{SectSDIS_Generalities}}

An issue that was left open in the previous section was that the estimator
that we constructed is biased from a practical standpoint. In this section, we
illustrate how to construct an efficient importance sampling estimator that
terminates in finite time and is unbiased. The estimator based on applying
state-independent importance sampling up until time $T_{\Delta}$ has been seen
to be inefficient. Examples of changes-of-measure that look reasonable from a
large deviations perspective but at the end turn out to have a poor
performance are well known in the rare-event simulation literature (see
Glasserman and Kou (1995)). It is interesting that estimating the tail of
$D_{\infty}$ provides yet another such example. These types of examples have
motivated the development of the theory behind the design of efficient
state-dependent importance sampling estimators, which is the basis behind the
construction of our estimator here. We shall explain some of the elements
behind this theory next.

We will follow the approach based on Lyapunov inequalities (see Blanchet and
Glynn (2008) and Blanchet, Glynn and Liu (2007)). Let us introduce some
notation for $W_{n}=\left(  S_{n},D_{n},X_{n}\right)  $. The transition kernel
associated to $W$ is denoted by $Q\left(  \cdot\right)  $, so%
\[
P_{w_{0}}\left(  W_{1}\in A\right)  =P\left(  \left.  W_{1}\in A\right\vert
W_{0}=w_{0}\right)  =\int_{A}Q\left(  w_{0},dw\right)  .
\]
A state-dependent importance sampling distribution for $W$ is described by the
Markov transition kernel%
\begin{equation}
Q_{r}\left(  w_{0},dw_{1}\right)  =r\left(  w_{0},w_{1}\right)  ^{-1}Q\left(
w_{0},dw_{1}\right)  , \label{KerR}%
\end{equation}
where $r\left(  \cdot\right)  $ is a positive function properly normalized so
that%
\[
\int Q_{r}\left(  w_{0},dw_{1}\right)  =\int r\left(  w_{0},w_{1}\right)
^{-1}Q\left(  w_{0},dw_{1}\right)  =1.
\]
The idea behind the Lyapunov method is to introduce a parametric family of
changes-of-measure. As we shall see, in our case, this will correspond to
suitably defined exponential changes-of-measure. This selection specifies
$r\left(  \cdot\right)  $. The associated importance sampling estimator, which
is obtained by sampling transitions from $Q_{r}\left(  \cdot\right)  $, takes
the form%
\[
Z_{\Delta}=r\left(  W_{0},W_{1}\right)  r\left(  W_{1},W_{2}\right)
...r\left(  W_{T_{\Delta}-1},W_{T_{\Delta}}\right)  I\left(  T_{\Delta}%
<\infty\right)  .
\]
Using $P_{w}^{\left(  r\right)  }\left(  \cdot\right)  $ (resp.\ $E_{w}%
^{\left(  r\right)  }\left(  \cdot\right)  $) to denote the probability
measure (resp.\ the expectation operator) induced by the transition kernel
$Q_{r}\left(  \cdot\right)  $ given that $W_{0}=w$, we can express the second
moment of $Z$ via%
\[
v_{\Delta}\left(  w\right)  =E_{w}^{\left(  r\right)  }Z_{\Delta}^{2}%
=E_{w}Z_{\Delta}.
\]
Note that conditioning on the first transition of the process $W$ one obtains%
\[
v_{\Delta}\left(  w\right)  =E_{w}[r\left(  w,W_{1}\right)  v_{\Delta}\left(
W_{1}\right)  ],
\]
subject to the boundary condition $v_{\Delta}\left(  w\right)  =1$ for
$w\in\mathbb{R}\times\lbrack1,\infty)\times\mathcal{S}$. We are interested in
a suitable upper bound for $v_{\Delta}\left(  w\right)  $, which can be
obtained by taking advantage of the following inequality proved in Blanchet
and Glynn (2008).

\begin{lemma}
\label{LemLI}If $h_{\Delta}\left(  \cdot\right)  $ is non-negative and
satisfies%
\begin{equation}
E_{w}[r\left(  w,W_{1}\right)  h_{\Delta}\left(  W_{1}\right)  ] \leq
h_{\Delta}\left(  w\right)  \label{LI}%
\end{equation}
subject to $h_{\Delta}\left(  w\right)  \geq1$ for $w\in\mathbb{R}%
\times\lbrack1,\infty)\times\mathcal{S}$, then
\[
v_{\Delta}\left(  w\right)  \leq h_{\Delta}\left(  w\right)  .
\]

\end{lemma}

\bigskip

Our strategy in state-dependent importance sampling is aligned with the
intuition behind the failure of the natural state-independent importance
sampling strategy described in the previous section; it consists in applying
importance sampling only when it is \textquotedblleft safe\textquotedblright%
\ to apply it. In other words, we wish to induce $T_{\Delta}<\infty$ by
exponentially tilting the increments of $S$, but we want to be careful and
maintain the likelihood ratio appropriately controlled. So, for instance,
cases where $D_{n}$ might be close to the boundary value $1$, but $S_{n}$ is
significantly smaller than $\log(1/\Delta)$ are of concern. In those cases, we
shall turn off importance sampling to avoid the accumulation of a large
likelihood ratio in $Z_{\Delta}$. In summary, suppose that the current
position of the cumulative discount rate process $S$ is given by $s$ and that
the position of the discounted process $D$ is $d$. We shall continue applying
exponential tilting as long as $(s,d,x)$ belongs to some region $C$ where it
is safe to apply importance sampling. We do not apply importance sampling if
$(s,d,x)\notin C$. The precise definition of the set $C$ will be given momentarily.

Using the notation introduced earlier leading to the statement of our Lyapunov
inequality in Lemma \ref{LemLI} we can describe the sampler as follows. Let
$C$ be an appropriately defined subset of $\mathbb{R\times R}\times
\mathcal{S}$. Assume that the current state of the process $W$ is
$w_{0}=(s_{0},d_{0},x_{0})$ and let us write $w_{1}=(s_{1},d_{1},y)$ for a
given outcome of the next transition. The function $r\left(  \cdot\right)
\triangleq r_{\theta_{\ast}}\left(  \cdot\right)  $ introduced in (\ref{KerR})
takes the form
\begin{align}
&  r_{\theta_{\ast}}^{-1}\left(  (s_{0},d_{0},x_{0}),(s_{1},d_{1},y)\right)
\label{Drth}\\
&  =I\left(  (s_{0},d_{0},x_{0})\in C\right)  \frac{u_{\theta_{\ast}}\left(
y\right)  }{u_{\theta_{\ast}}\left(  x_{0}\right)  }\exp\left(  \theta_{\ast
}(s_{1}-s_{0})\right) \nonumber\\
&  +I\left(  (s_{0},d_{0},x_{0})\notin C\right)  .\nonumber
\end{align}

The construction of an appropriate Lyapunov function $h_{\Delta}\left(
\cdot\right)  $ involves applying Lemma \ref{LemLI}. In turn, the definition
of the set $C$ is coupled with the construction of $h_{\Delta}\left(
\cdot\right)  $. We shall construct $h_{\Delta}\left(  \cdot\right)  $ so that
$h_{\Delta}\left(  s,d,x\right)  =1$ implies $(s,d,x)\notin C$. Moreover, we
shall impose the condition $h_{\Delta}\left(  \cdot\right)  \in\lbrack0,1]$.
Assuming $h_{\Delta}$ can be constructed in this way we immediately have that
the Lyapunov inequality is satisfied outside $C$. We then need to construct
$h_{\Delta}$ on $C$. We wish to find an asymptotically optimal
change-of-measure, so it makes sense to propose
\[
h_{\Delta}\left(  s,d,x\right)  =O(P_{\left(  s,d,x\right)  }\left(
T_{\Delta}<\infty\right)  ^{2-\rho_{\Delta}}),
\]
where $\rho_{\Delta}\searrow0$ as $\Delta\searrow0$ (recall the definition of
asymptotic optimality given in the Introduction). On the other hand, we have
that%
\begin{align}
P_{\left(  s,d,x\right)  }\left(  T_{\Delta}<\infty\right)   &  =P_{(0,0,x)}%
\left(  d+\exp\left(  s\right)  \Delta D_{\infty}>1\right) \nonumber\\
&  =P_{(0,0,x)}\left(  D_{\infty}>\exp\left(  -s\right)  \left(  \frac
{1-d}{\Delta}\right)  \right) \nonumber\\
&  \approx\exp\left(  s\theta_{\ast}\right)  [\Delta/(1-d)]^{\theta_{\ast}}.
\label{App1}%
\end{align}
Motivated by the form of this approximation, which is expected to hold at
least in logarithmic sense as $\exp(-s)[(1-d)/\Delta]\longrightarrow\infty$,
we suggest a Lyapunov function of the form
\[
h_{\Delta}\left(  s,d,x\right)  =\min\{c_{\Delta}^{2\theta_{\ast}-\rho
_{\Delta}}\exp\left(  [2\theta_{\ast}-\rho_{\Delta}]s\right)  [\Delta
/(1-d)_{+}]^{2\theta_{\ast}-\rho_{\Delta}}u_{\theta_{\ast}}\left(  x\right)
u_{\theta_{\ast}-\rho_{\Delta}}\left(  x\right)  ,1\}.
\]
The introduction of the function $u_{\theta_{\ast}}\left(  x\right)
u_{\theta_{\ast}-\rho_{\Delta}}\left(  x\right)  $ in a multiplicative form as
given above is convenient for the purpose of verifying Lyapunov inequalities
for importance sampling in the setting of Markov random walks (see Blanchet,
Leder and Glynn (2009)). The constant $c_{\Delta}>0$, which will be specified
in the verification of the Lyapunov inequality, is introduced as an extra
degree of freedom to recognize that approximation (\ref{App1}) may not be
exact. The exponent on top of $c_{\Delta}$ allows to make the estimates in the
verification of the Lyapunov inequality somewhat cleaner. Note that we have
$h_{\Delta}\left(  \cdot\right)  \in\lbrack0,1]$ and the set $C$ is defined
via%
\begin{equation}
C=\{(s,d,x):h_{\Delta}\left(  s,d,x\right)  <1\}. \label{DefC}%
\end{equation}
We do not apply importance sampling whenever we reach a state $(s,d,x)$
satisfying $h_{\Delta}\left(  s,d,x\right)  =1$.

We shall close this section with a precise description of our state-dependent
algorithm. The following procedure generates one sample of our estimator.

\textbf{State-Dependent Algorithm}

\begin{enumerate}
\item[Step 1:] Set $\rho_{\Delta}=1/\log(1/\Delta)$ and $c_{\Delta}%
=(B_{2}/B_{1})\rho_{\Delta}^{-(1+1/(2\theta^{\ast}-\rho_{\Delta}))}$ with
$0<B_{1},B_{2}<\infty$ as indicated in Proposition \ref{PropLIa} below.
Initialize $(s,d,x)\leftarrow(0,0,x_{0})$.

\item[Step 2:] Initialize likelihood ratio $L\leftarrow1$.

\item[Step 3:] While at $(s,d,x)$, do the following:

\begin{enumerate}
\item[1.] If $(s,d,x)\in C$ defined in \eqref{DefC} i.e. $h_{\Delta}(s,d,x)<1$,

\begin{enumerate}
\item Generate $X_{1}$ from the kernel
\[
K_{\theta_{\ast}}(x,y)=K(x,y)\exp(\chi(y,\theta_{\ast})-\psi(\theta_{\ast
}))\frac{u_{\theta_{\ast}}(y)}{u_{\theta_{\ast}}(x)}.
\]
Say we have realization $X_{1}=y$.

\item Given $X_{1}=y$, sample $\gamma(y,\xi_{1})$ from the exponential
tilting
\[
P^{\theta_{\ast}}(\gamma(y,\xi_{1})\in dz)=\exp(\theta_{\ast}z-\chi
(y,\theta_{\ast}))P(\gamma(y,\xi_{1})\in dz).
\]
Say we have $\gamma(y,\xi_{1})=z$.

\item Sample $\lambda(y,\eta_{1})$ from the nominal distribution of $\eta_{1}$
given $X_{1}=y$ and $\gamma\left(  y,\xi_{1}\right)  =z$.\newline Say we have
$\lambda(y,\eta_{1})=w$.

\item Update
\[
L\leftarrow L\times\exp(-\theta_{\ast}z)\frac{u_{\theta_{\ast}}(x)}%
{u_{\theta_{\ast}}(y)}.
\]

\end{enumerate}

Else if $(s,d,x)\notin C$ i.e. $h_{\Delta}(s,d,x)=1$,

\begin{enumerate}
\item Sample $X_{1}$ from its nominal distribution. Given $X_{1}=y$, sample
$\gamma(y,\xi_{1})$ and $\lambda(y,\eta_{1})$ from their nominal
distributions. Say the realizations are $\gamma(y,\xi_{1})=y$ and
$\lambda(y,\eta_{1})=w$.
\end{enumerate}

\item[2.] Update
\[
(s,d,x)\leftarrow(s+z,d+\Delta w\exp(s+z),y).
\]

\item[3.] If $d>1$, output $L$ and stop; else repeat the loop.
\end{enumerate}
\end{enumerate}

\bigskip

The variance analysis of the unbiased estimator $L$ as well as the termination
time of the algorithm are given in the next sections.

\section{Efficiency of State-Dependent Importance
Sampling\label{SectSDIS_Efficiency}}

In order to verify asymptotic optimality of $L$ we first show that $h_{\Delta
}\left(  \cdot\right)  $ satisfies the Lyapunov inequality given in Lemma
\ref{LemLI}. We have indicated that the inequality is satisfied outside $C$.
On the other hand, one clearly has that $h_{\Delta}\left(  s,d,x\right)  =1$
for $d\geq1$, so the boundary condition given in Lemma \ref{LemLI} is
satisfied. Consequently, in order to show that $h_{\Delta}\left(
\cdot\right)  $ is a valid Lyapunov function and that $v_{\Delta}\left(
w\right)  \leq h_{\Delta}\left(  w\right)  $ we just have to show the
following proposition.

%\newline

\begin{proposition}
\label{PropLIa}Suppose that Assumptions 1 to 3 are in force and select
$b_{0}<\infty$ such that $0<1/[\inf_{\theta\in\left(  0,\theta_{\ast}\right)
x\in\mathcal{S}}u_{\theta}(x)]^{2}\leq b_{0}$.
\end{proposition}

\textit{i) Select }$b_{1}>0$\textit{ such that for each }$\delta\in
(0,\theta_{\ast})$\textit{ }%
\[
\exp\left(  \psi\left(  \theta_{\ast}-\delta\right)  \right)  \leq1-\delta
\mu+b_{1}\delta^{2}%
\]
\textit{where }$\mu=d\psi(\theta_{\ast})/d\theta>0$\textit{.}

\textit{ii) Pick }$b_{2}\in(0,\infty)$\textit{ such that}%
\[
\sup_{x\in\mathcal{S},\delta\in(0,\theta_{*})}E_{x}[\lambda(X_{1},\eta
_{1})^{2\theta_{\ast}-\delta}\exp((\theta_{\ast}-\delta)\gamma(X_{1},\xi
_{1}))]\leq b_{2}%
\]

\textit{and make the following selection of $B_{1}$, $B_{2}$, $\rho_{\Delta}$
and $c_{\Delta}$:}

\textit{iii) Select }$0<B_{1},B_{2}<\infty$\textit{ and }$\rho_{\Delta
},c_{\Delta}>0$\textit{ so that }$\rho_{\Delta}\searrow0$, $\rho_{\Delta}%
\in(0,\theta_{\ast})$, $c_{\Delta}=(B_{2}/B_{1})\rho_{\Delta}^{-(1+1/(2\theta
_{\ast}-\rho_{\Delta}))}$, $B_{1}\rho_{\Delta}<1$ \textit{and}\newline%
\[
\frac{b_{0}b_{2}\rho_{\Delta}}{B_{2}^{2\theta_{\ast}-\rho_{\Delta}}}%
+\frac{(1-\rho_{\Delta}\mu+b_{1}\rho_{\Delta}^{2})}{(1-B_{1}\rho_{\Delta
})^{(2\theta_{\ast}-\rho_{\Delta})}}\leq1.
\]

\textit{Then, }$h_{\Delta}\left(  \cdot\right)  $\textit{ satisfies the
Lyapunov inequality (\ref{LI}) on }$C$\textit{ assuming that }$r\left(
\cdot\right)  =r_{\theta_{\ast}}\left(  \cdot\right)  $\textit{ is given as in
(\ref{Drth}).}

\begin{proof}
First, $b_{0}$ is finite because $u_{\theta_{\ast}}(\cdot)$ is strictly
positive and $\mathcal{S}$ is finite. The fact that the selections in i) and
iii) are always possible follows from straightforward Taylor series
developments. The selection of $b_{2}$ in ii) is possible because of
Assumptions 2 and 3 combined with Holder's inequality. We shall prove the
concluding statement in the result using the selected values in i), ii) and
iii). Assume that $(s,d,x)$ is such that $h_{\Delta}\left(  s,d,x\right)  <1$.
To ease the notation we write $\gamma_{1}=\gamma\left(  X_{1},\xi_{1}\right)
$ and $\lambda_{1}=\lambda\left(  X_{1},\eta_{1}\right)  $. We need to show
that%
\begin{align}
&  E_{x}h_{\Delta}\left(  s+\gamma_{1},d+\exp\left(  s+\gamma_{1}\right)
\Delta\lambda_{1},X_{1}\right)  L_{\theta_{\ast}}[X_{1},\gamma_{1}%
]\label{E1}\\
&  \leq h_{\Delta}\left(  s,d,x\right)  ,\nonumber
\end{align}
where%
\[
L_{\theta_{\ast}}[X_{1},\gamma_{1}]=\frac{u_{\theta_{\ast}}\left(  x\right)
}{u_{\theta_{\ast}}\left(  X_{1}\right)  }\exp\left(  -\theta_{\ast}\gamma
_{1}\right)  .
\]
We divide the expectation in (\ref{E1}) into two parts, namely, transitions
that lie in a region that corresponds to the complement of $C$ and transitions
that lie within $C$. To be precise, set $a_{\Delta}\in(0,1)$ and put%
\[
A=\{\exp\left(  \gamma_{1}\right)  \lambda_{1}\geq a_{\Delta}\exp\left(
-s\right)  (1-d)/\Delta\}
\]
and write $A^{c}$ for the complement of $A$. The expectation in (\ref{E1}) is
then equal to $J_{1}+J_{2}$, where%
\begin{align*}
J_{1}  &  =E_{x}(h_{\Delta}\left(  s+\gamma_{1},d+\exp\left(  s+\gamma
_{1}\right)  \Delta\lambda_{1},X_{1}\right)  L_{\theta_{\ast}}[X_{1}%
,\gamma_{1}];A),\\
J_{2}  &  =E_{x}(h_{\Delta}\left(  s+\gamma_{1},d+\exp\left(  s+\gamma
_{1}\right)  \Delta\lambda_{1},X_{1}\right)  L_{\theta_{\ast}}[X_{1}%
,\gamma_{1}];A^{c}).
\end{align*}
We first analyze $J_{1}/h_{\Delta}\left(  s,d,x\right)  $. Note that%
\begin{align*}
\frac{J_{1}}{h_{\Delta}\left(  s,d,x\right)  }  &  \leq\frac{E_{x}%
(L_{\theta_{\ast}}[X_{1},\gamma_{1}];A)}{h_{\Delta}\left(  s,d,x\right)  }\\
&  \leq\frac{u_{\theta_{\ast}}\left(  x\right)  }{\inf_{y\in\mathcal{S}%
}u_{\theta_{\ast}}\left(  y\right)  }\times\frac{E_{x}[\exp\left(
-\theta_{\ast}\gamma_{1}\right)  ;A]}{h_{\Delta}\left(  s,d,x\right)  }.
\end{align*}
Now we note that%
\begin{align*}
&  E_{x}\left[  \exp\left(  -\theta_{\ast}\gamma_{1}\right)  ;A\right] \\
&  =E_{x}(\exp\left(  -\theta_{\ast}\gamma_{1}\right)  ;\exp\left(
-\gamma_{1}\right)  \leq\lambda_{1}\Delta\exp(s)/[(1-d)a_{\Delta}])\\
&  \leq\left(  \frac{\Delta}{1-d}\right)  ^{\theta_{\ast}}\frac{\exp
(\theta_{\ast}s)}{a_{\Delta}^{\theta_{\ast}}}\\
&  \times E_{x}\left[  \lambda_{1}^{\theta_{\ast}};\lambda_{1}\geq\exp
(-\gamma_{1}-s)a_{\Delta}(1-d)/\Delta\right]  .
\end{align*}
Moreover, applying Markov's inequality we obtain that for each $\beta>0$%
\begin{align*}
&  E_{x}\left[  \lambda_{1}^{\theta_{\ast}};\lambda_{1}\Delta\exp(\gamma
_{1}+s)/[a_{\Delta}(1-d)]\geq1\right] \\
&  \leq\frac{\Delta^{\beta}\exp(\beta s)}{[a_{\Delta}(1-d)]^{\beta}}%
E_{x}\left[  \lambda_{1}^{\theta_{\ast}+\beta}\exp(\beta\gamma_{1})\right]  .
\end{align*}
Selecting $\beta=\theta_{\ast}-\rho_{\Delta}$ we obtain (using ii))%
\begin{align*}
\frac{J_{1}}{h_{\Delta}\left(  s,d,x\right)  }  &  \leq\left(  \frac{\Delta
}{1-d}\right)  ^{2\theta_{\ast}-\rho_{\Delta}}\frac{\exp(\left(  2\theta
_{\ast}-\rho_{\Delta}\right)  s)}{a_{\Delta}^{2\theta_{\ast}-\rho_{\Delta}}%
}\frac{u_{\theta_{\ast}}\left(  x\right)  E_{x}[\lambda_{1}^{2\theta_{\ast
}-\rho_{\Delta}}\exp((\theta_{\ast}-\rho_{\Delta})\gamma_{1})]}{\inf
_{y\in\mathcal{S}}u_{\theta_{\ast}}\left(  y\right)  h_{\Delta}\left(
s,d,x\right)  }\\
&  \leq\frac{b_{2}}{\left(  a_{\Delta}c_{\Delta}\right)  ^{2\theta_{\ast}%
-\rho_{\Delta}}\inf_{y\in\mathcal{S}}u_{\theta_{\ast}}\left(  y\right)
\inf_{y\in\mathcal{S}}u_{\theta_{\ast}-\rho_{\Delta}}\left(  y\right)  }%
\leq\frac{b_{2}b_{0}}{\left(  a_{\Delta}c_{\Delta}\right)  ^{2\theta_{\ast
}-\rho_{\Delta}}}.
\end{align*}
To analyze $J_{2}$ we note that on%
\[
A^{c}=\{\exp(\gamma_{1}+s)\lambda_{1}\Delta/(1-d)<a_{\Delta}\}
\]
we have that%
\begin{align*}
&  \frac{h_{\Delta}\left(  s+\gamma_{1},d+\exp\left(  s+\gamma_{1}\right)
\Delta\lambda_{1},X_{1}\right)  }{h_{\Delta}\left(  s,d,x\right)  }\\
&  =\exp((2\theta_{\ast}-\rho_{\Delta})\gamma_{1})\left(  1-\frac{\Delta
\exp(s+\gamma_{1})\lambda_{1}}{1-d}\right)  ^{-(2\theta_{\ast}-\rho_{\Delta}%
)}\frac{u_{\theta_{\ast}}\left(  X_{1}\right)  u_{\theta_{\ast}-\rho_{\Delta}%
}\left(  X_{1}\right)  }{u_{\theta_{\ast}}\left(  x\right)  u_{\theta_{\ast
}-\rho_{\Delta}}\left(  x\right)  }\\
&  \leq\exp((2\theta_{\ast}-\rho_{\Delta})\gamma_{1})\left(  1-a_{\Delta
}\right)  ^{-(2\theta_{\ast}-\rho_{\Delta})}\frac{u_{\theta_{\ast}}\left(
X_{1}\right)  u_{\theta_{\ast}-\rho_{\Delta}}\left(  X_{1}\right)  }%
{u_{\theta_{\ast}}\left(  x\right)  u_{\theta_{\ast}-\rho_{\Delta}}\left(
x\right)  }.
\end{align*}
Therefore, we have that%
\begin{align*}
&  \frac{J_{2}}{h\left(  s,d,x\right)  }\\
&  \leq(1-a_{\Delta})^{-(2\theta_{\ast}-\rho_{\Delta})}E_{x}\left(
\exp((2\theta_{\ast}-\rho_{\Delta})\gamma_{1})\frac{u_{\theta_{\ast}}\left(
X_{1}\right)  u_{\theta_{\ast}-\rho_{\Delta}}\left(  X_{1}\right)  }%
{u_{\theta_{\ast}}\left(  x\right)  u_{\theta_{\ast}-\rho_{\Delta}}\left(
x\right)  }L_{\theta_{\ast}}[X_{1},\gamma_{1}]\right) \\
&  =(1-a_{\Delta})^{-(2\theta_{\ast}-\rho_{\Delta})}E_{x}\left(  \exp
((\theta_{\ast}-\rho_{\Delta})\gamma_{1})\frac{u_{\theta_{\ast}-\rho_{\Delta}%
}\left(  X_{1}\right)  }{u_{\theta_{\ast}-\rho_{\Delta}}\left(  x\right)
}\right) \\
&  =(1-a_{\Delta})^{-(2\theta_{\ast}-\rho_{\Delta})}\exp\left(  \psi\left(
\theta_{\ast}-\rho_{\Delta}\right)  \right)  .
\end{align*}
Recall that we have assumed in ii) that $b_{1}$ is selected so that%
\[
\exp\left(  \psi\left(  \theta_{\ast}-\rho_{\Delta}\right)  \right)
\leq1-\rho_{\Delta}\mu+b_{1}\rho_{\Delta}^{2}.
\]
Thus we obtain%
\[
\frac{J_{2}}{h\left(  s,d,x\right)  }\leq(1-a_{\Delta})^{-(2\theta_{\ast}%
-\rho_{\Delta})}(1-\rho_{\Delta}\mu+b_{1}\rho_{\Delta}^{2}).
\]
Combining our estimates for $J_{1}$ and $J_{2}$ together we arrive at%
\begin{align*}
&  \frac{J_{1}}{h\left(  s,d,x\right)  }+\frac{J_{2}}{h\left(  s,d,x\right)
}\\
&  \leq\frac{b_{0}b_{2}}{\left(  a_{\Delta}c_{\Delta}\right)  ^{2\theta_{\ast
}-\rho_{\Delta}}}+\frac{(1-\rho_{\Delta}\mu+b_{1}\rho_{\Delta}^{2}%
)}{(1-a_{\Delta})^{(2\theta_{\ast}-\rho_{\Delta})}}.
\end{align*}
Let $a_{\Delta}=B_{1}\rho_{\Delta}$ and $c_{\Delta}=(B_{2}/B_{1})\rho_{\Delta
}^{-(1+1/(2\theta^{\ast}-\rho_{\Delta}))}$ for $\rho_{\Delta}<\theta_{\ast}$,
substitute in the previous inequality and conclude that%
\begin{align*}
&  \frac{J_{1}}{h\left(  s,d,x\right)  }+\frac{J_{2}}{h\left(  s,d,x\right)
}\\
&  \leq\frac{b_{0}b_{2}\rho_{\Delta}}{B_{2}^{2\theta_{\ast}-\rho_{\Delta}}%
}+\frac{(1-\rho_{\Delta}\mu+b_{1}\rho_{\Delta}^{2})}{(1-B_{1}\rho_{\Delta
})^{(2\theta_{\ast}-\rho_{\Delta})}}\leq1,
\end{align*}
where the previous inequality follows by the selection of $B_{1}$ and $B_{2}$
in Assumption iii). This concludes the proof of the proposition.
\end{proof}

\bigskip

The next result summarizes the asymptotic optimality properties of the
algorithm obtained out of the previous development.

\begin{theorem}
\label{ThmMain1}Select $\rho_{\Delta}=1/\log(1/\Delta)$ and $c_{\Delta}%
=(B_{2}/B_{1})\rho_{\Delta}^{-(1+1/(2\theta^{\ast}-\rho_{\Delta}))}$ with
$0<B_{1},B_{2}<\infty$ as indicated in Proposition \ref{PropLIa}. Then, the
resulting estimator $L$ obtained by the State-dependent Algorithm has a
coefficient of variation of order $O(c_{\Delta}^{2\theta_{\ast}})$ and
therefore, in particular, it is asymptotically optimal.
\end{theorem}

\begin{proof}
The result follows as an immediate consequence of the fact that $h_{\Delta}$
is a valid Lyapunov function combined with Theorem \ref{ThmAsympt}.
\end{proof}

\bigskip

\section{Unbiasedness and Logarithmic Running Time of State-Dependent
Sampler\label{SectSDIS_Unbias_RunTime}}

Throughout the rest of our development, in addition to Assumptions 1 to 3, we
impose the following mild technical assumption.

\bigskip

\textbf{Assumption 4: }For each $x\in\mathcal{S}$, $Var\left(  \gamma\left(
x,\xi\right)  \right)  >0$.

\bigskip

The previous assumption simply says that $\gamma\left(  x,\xi\right)  $ is
random. The assumption is immediately satisfied (given Assumption 2) in the
i.i.d. case. As we shall explain a major component in our algorithm is the
construction of a specific path that leads to termination. Assumption 4 is
imposed in order to rule out a cyclic type behavior under the importance
sampling distribution.

We shall show that the state-dependent algorithm stops with probability one
(and hence avoids artificial termination which causes bias, a potential
problem with the algorithm in Section 3.3) and that the expected termination
time is of order $(\log(1/\Delta))^{p}$ for some $p>0$. To do so let us
introduce some convenient notation. Let $Z_{n}\triangleq(1-D_{n})e^{-S_{n}}$,
and put $Y_{n}=(X_{n},Z_{n})$ for $n\geq0$. The dynamics of the process
$Y=(Y_{n}:n\geq0)$ are such that
\[
Y_{n+1}=(X_{n+1},Z_{n}e^{-\gamma(X_{n+1},\xi_{n+1})}-\Delta\lambda
(X_{n+1},\eta_{n+1})).
\]
It is then easy to see that our state-dependent algorithm, which
mathematically is described by equations (\ref{Drth}) and (\ref{DefC}), can be
stated in the following way in terms of $Y_{n}$: Given $Y_{n}$,

\begin{itemize}
\item Apply exponential tilting to $\gamma(X_{n+1},\xi_{n+1})$ (using the
tilting parameter $\theta_{\ast}$) if
\[
Z_{n}>c_{\Delta}\Delta u_{\theta_{*}}(X_{n})^{1/(2\theta_{\ast}-\rho_{\Delta
})}u_{\theta_{*}-\rho_{\Delta}}(X_{n})^{1/(2\theta_{*}-\rho_{\Delta})}%
\]
.

\item Transition according to the nominal distribution of the system if
\[
0<Z_{n}\leq c_{\Delta}\Delta u_{\theta_{*}}(X_{n})^{1/(2\theta_{\ast}%
-\rho_{\Delta})}u_{\theta_{*}-\rho_{\Delta}}(X_{n})^{1/(2\theta_{*}%
-\rho_{\Delta})}.
\]

\item Terminate if\ $Z_{n}\leq0$.
\end{itemize}

Note that the region when $Z_{n}>c_{\Delta}\Delta u_{\theta_{*}}%
(X_{n})^{1/(2\theta_{\ast}-\rho_{\Delta})}u_{\theta_{*}-\rho_{\Delta}}%
(X_{n})^{1/(2\theta_{*}-\rho_{\Delta})}$ corresponds to the region $C$ (recall
equation (\ref{DefC}) in Section 4). If $0<Z_{n}\leq c_{\Delta}\Delta
u_{\theta_{*}}(X_{n})^{1/(2\theta_{\ast}-\rho_{\Delta})}u_{\theta_{*}%
-\rho_{\Delta}}(X_{n})^{1/(2\theta_{*}-\rho_{\Delta})}$ we say that $Y_{n}$ is
in $C^{\prime}$. Finally, we say that $Y_{n}$ is in $B$ if $Z_{n}\leq0$. In
other words, the set $B$ is the termination set. Let us write $m=\inf
_{x\in\mathcal{S}}\{u_{\theta_{*}}(x)^{1/(2\theta_{\ast}-\rho_{\Delta}%
)}u_{\theta_{*}-\rho_{\Delta}}(x)^{1/(2\theta_{*}-\rho_{\Delta})}\}$ and
$M=\sup_{x\in\mathcal{S}}\{u_{\theta_{*}}(x)^{1/(2\theta_{\ast}-\rho_{\Delta
})}u_{\theta_{*}-\rho_{\Delta}}(x)^{1/(2\theta_{*}-\rho_{\Delta})}\}$. Note
that $0<m<M<\infty$. A key observation is that the set $C^{\prime}$ is
bounded. This will help in providing bounds for the running time of the
algorithm as we shall see.

We will obtain an upper bound for the algorithm by bounding the time spent by
the process in both regions $C$ and $C^{\prime}$. Intuitively, starting from
an initial position in $C$, the process moves to $C^{\prime}$ in finite number
of steps. Then the process moves to either $C$ or $B$. If the process enters
$C$ before $B$, then from $C$ it again moves back to $C^{\prime}$ and the
iteration between region $C^{\prime}$ and $C$ repeats until the process
finally hits $B$, which is guaranteed to happen by geometric trial argument.
Our proof below will make the intuition rigorous and shows that the time for
the process to travel each back-and-forth between $C$ and $C^{\prime}$ is
logarithmic in $1/\Delta$, and that there is a significant probability that
the process starting from $C^{\prime}$ hits $B$ before $C$. This, overall,
will imply a logarithmic running time of the algorithm. More precisely, we
will show the following lemmas. The reader should keep in mind the selections
\[
\rho_{\Delta}=1/\log(1/\Delta)\text{ \ and \ }c_{\Delta}=(B_{2}/B_{1}%
)\rho_{\Delta}^{-(1+1/(2\theta^{\ast}-\rho_{\Delta}))}=O\left(  \log
(1/\Delta)^{1+1/(2\theta^{\ast})}\right)
\]
given in Theorem \ref{ThmMain1}.

Recall the notations $P^{\theta_{*}}(\cdot)$ and $E^{\theta_{*}}[\cdot]$ to
denote the probability measure and expectation under the state-dependent
importance sampler. Note that under $P^{\theta_{*}}(\cdot)$ no exponential
tilting is performed when the current state $Y_{n}$ lies in $C^{\prime}$.

\begin{lemma}
Denote $T_{C\cup B}=\inf\{n>0:Y_{n}\in C\cup B\}$. Under Assumptions 1 to 4 we
have
\[
E_{y_{0}}^{\theta_{\ast}}[T_{C\cup B}]=O(c_{\Delta}^{p}\log c_{\Delta})
\]
uniformly over $y_{0}\in C^{\prime}$ for some constant $p>0$. \label{lemma2}
\end{lemma}

\begin{lemma}
Let $T_{C}=\inf\{n>0:Y_{n}\in C\}$ and $T_{B}=\inf\{n>0:Y_{n}\in B\}$. If
Assumptions 1 to 4 hold, then, uniformly over $y_{0}=(x_{0},z_{0})\in
C^{\prime}$,
\[
P_{y_{0}}^{\theta_{\ast}}(T_{B}<T_{C})\geq\frac{c_{1}}{c_{\Delta}^{p}}%
\]
for some constant $c_{1}$ and $p$ (the $p$ can be chosen as the same $p$ in
Lemma \ref{lemma2}). \label{lemma1}
\end{lemma}

\begin{lemma}
Denote $T_{C^{\prime}\cup B}=\inf\{n>0:Y_{n}\in C^{\prime}\cup B\}$ and
suppose that Assumptions 1 to 4 are in force. For any $y_{0}=(x_{0},z_{0})\in
C$, we have $P_{y_{0}}^{\theta_{\ast}}(T_{C^{\prime}\cup B}<\infty)=1$ and
\[
E_{y_{0}}^{\theta_{\ast}}[T_{C^{\prime}\cup B}]=O\left(  \log\left(
\frac{z_{0}}{mc_{\Delta}\Delta}\right)  \right)
\]
\label{lemma3}
\end{lemma}

The first lemma shows that it takes on average a logarithmic number of steps
(in $1/\Delta$) for the process to reach either $B$ or $C$ from $C^{\prime}$.
The second lemma shows that there is a significant probability, uniformly over
the initial positions in $C^{\prime}$, that the process reaches $B$ before
$C$. The third lemma states that the time taken from $C$ to $C^{\prime}$ is
also logarithmic in $1/\Delta$. Lemmas \ref{lemma2} and \ref{lemma3} guarantee
that each cross-border travel, either from $C$ to $C^{\prime}$ or from
$C^{\prime}$ to $C$, takes on average logarithmic time. On the other hand,
Lemma \ref{lemma1} guarantees that a geometric number of iteration, with
significant probability of success, will bring the process to $B$ from some
state in $C^{\prime}$. These will prove the following proposition on the
algorithmic running time.

\begin{proposition}
Suppose that Assumptions 1 to 4 hold. Then, for any $y_{0}=(x_{0},z_{0})$, we
have $P_{y_{0}}^{\theta_{\ast}}(T_{B}<\infty)=1$ and
\[
E_{y_{0}}^{\theta_{\ast}}[T_{B}]=O\left(  c_{\Delta}^{p}\log c_{\Delta}%
+\log\left(  \frac{1}{c_{\Delta}\Delta}\right)  \right)
\]
for some $p>0$. \label{termination}
\end{proposition}

\bigskip

We now give the proofs of the lemmas and Proposition \ref{termination}.

\bigskip

\begin{proof}
[Proof of Lemma \ref{lemma2}]Our strategy to prove Lemma \ref{lemma2} is the
following. Given any initial state $y_{0}=(x_{0},z_{0})\in C^{\prime}$, we
first construct explicitly a path that takes $y_{0}$ to $B$ within $O(\log
c_{\Delta})$ steps (which we call event $A(x_{0})$ below), and we argue that
this path happens with probability $\Omega(c_{\Delta}^{-p})$ for some $p>0$.
Then we look at the process in blocks of $O(\log c_{\Delta})$ steps. For each
block, if the process follows the particular path that leads to $B$, then
$T_{B}$, and hence $T_{B\cup C}$, is hit; otherwise the process may have hit
$C$, or may continue to the next block starting with some state in $C^{\prime
}$. In other words, $T_{B\cup C}$ is bounded by the time from the initial
position up to the time that the process finishes following exactly the
particular path in a block. We note that a successful follow of the particular
path is a geometric r.v. with parameter $O(c_{\Delta}^{-p})$, and hence the
mean of $T_{B\cup C}$ is bounded by $O(c_{\Delta}^{-p})\times O(\log
c_{\Delta})$ and therefore the result. Now we make the previous intuition rigorous.

We first prove some elementary probabilistic bounds for $\gamma(\cdot)$ and
$\lambda(\cdot)$. As in Section 4 we simplify our notation by writing
$\gamma_{n}=\gamma(X_{n},\xi_{n})$ and $\lambda_{n}=\lambda(X_{n},\xi_{n})$.
We first argue that for any $y=(x,z)\in C^{\prime}$,
\begin{equation}
P_{y}^{\theta_{*}}\left(  e^{-\gamma_{1}}\frac{u_{\theta_{\ast}}%
(x)^{1/\theta_{\ast}}}{u_{\theta_{\ast}}(X_{1})^{1/\theta_{\ast}}}\leq
u_{1}\right)  >0 \label{prob1}%
\end{equation}
for some $0<u_{1}<1$. Note that the initial conditioning in the probability in
\eqref{prob1} depends on $y$ only through $x$.

We prove \eqref{prob1} by contradiction. Suppose \eqref{prob1} is not true,
then there exists some Markov state $w$ such that
\[
P_{w}^{\theta_{*}}\left(  e^{-\gamma_{1}}\frac{u_{\theta_{\ast}}%
(w)^{1/\theta_{\ast}}}{u_{\theta_{\ast}}(X_{1})^{1/\theta_{\ast}}}%
\geq1\right)  =1.
\]
Now if this happens and additionally
\[
P_{w}^{\theta_{*}}\left(  e^{-\gamma_{1}}\frac{u_{\theta_{\ast}}%
(w)^{1/\theta_{\ast}}}{u_{\theta_{\ast}}(X_{1})^{1/\theta_{\ast}}}>1\right)
>0,
\]
which obviously implies
\[
P_{w}^{\theta_{*}}\left(  e^{\gamma_{1}}\frac{u_{\theta_{\ast}}(X_{1}%
)^{1/\theta_{\ast}}}{u_{\theta_{\ast}}(w)^{1/\theta_{\ast}}}<1\right)  >0,
\]
then
\[
E_{w}^{\theta_{*}}\left[  e^{\theta_{\ast}\gamma_{1}}\frac{u_{\theta_{\ast}%
}(X_{1})}{u_{\theta_{\ast}}(w)}\right]  <1,
\]
which contradicts the definition of $\theta_{\ast}$. Hence we are left with
the possibility that
\[
P_{w}^{\theta_{*}}\left(  e^{-\gamma_{1}}\frac{u_{\theta_{\ast}}%
(w)^{1/\theta_{\ast}}}{u_{\theta_{\ast}}(X_{1})^{1/\theta_{\ast}}}=1\right)
=1,
\]
but this contradicts our non-degeneracy assumption, namely, Assumption 4.

Using \eqref{prob1}, note that we can pick $u_{2}>0$ small enough such that
\begin{equation}
P_{y}^{\theta_{*}}\left(  e^{-\gamma_{1}}\frac{u_{\theta_{*}}(x)^{1/\theta
_{*}}}{u_{\theta_{*}}(X_{1})^{1/\theta_{*}}}\leq u_{1},\ e^{-\gamma_{1}}\geq
u_{2}\right)  >\epsilon_{1}>0 \label{prob3}%
\end{equation}
for any $y=(x,z)\in C^{\prime}$. This follows from a contradiction proof since
the non-existence of $u_{2}$ would imply $e^{-\gamma_{1}}=0$ a.s.

On the other hand, it is easy to see that there exists $r_{1}$ and $r_{2}$ and
a small enough $u_{3}>0$ such that
\begin{equation}
P_{r_{1}}^{\theta_{*}}(X_{1}=r_{2},\ \lambda(r_{4},\eta_{1})\geq
u_{3})>\epsilon_{2}>0 \label{prob5}%
\end{equation}
since otherwise $\lambda_{1}=0$ a.s.

We will now construct the path $A(x_{0})$ as discussed earlier in the proof.
This path will depend on the initial position $y_{0}=(x_{0},z_{0})\in
C^{\prime}$, but it has length $O(\log c_{\Delta})$ uniformly over any initial
position in $C^{\prime}$. The path has the property that whenever the Markov
state hits $r_{1}$, it would go to $r_{2}$ with $\lambda(r_{2},\eta_{n})\geq
u_{3}$ in the next state. Moreover, for every step, $e^{-\gamma(X_{n},\xi
_{n})}u_{\theta_{*}}(X_{n-1})^{1/\theta_{*}}/u_{\theta_{*}}(X_{n}%
)^{1/\theta_{*}}\leq u_{1}$ and $e^{-\gamma_{n}}\geq u_{2}$. The path evolves
in a periodic way i.e. it hits $r_{1}$ in every $l$ steps for $N$ times, where
$N$ is a number to be determined later. The existence of $l$ and the
occurrence of such periodic cycles with positive probability is guaranteed by
the irreducibility of the Markov chain $X_{n}$. In other words, consider the
event $A(x_{0})$ given by
\begin{align*}
A(x_{0})  &  =\Bigg\{\text{\ for\ }k=1,\ldots,N,\ X_{a+kl}=r_{1}%
,\ X_{a+kl+1}=r_{2},\ \lambda(r_{2},\eta_{a+kl+1})\geq u_{3};{}\\
&  {}\text{\ for\ }i=1,\ldots,a+Nl+b,\ e^{-\gamma(X_{i},\xi_{i})}%
\frac{u_{\theta_{*}}(X_{i-1})^{1/\theta_{*}}}{u_{\theta_{*}}(X_{i}%
)^{1/\theta_{*}}}\leq u_{1},\ e^{-\gamma(X_{i},\xi_{i})}\geq u_{2};{}\\
&  {}X_{a+Nl+b}=x_{0}\Bigg\}
\end{align*}
where $a$ is the number of steps for the initial state $x_{0}$ to reach
$r_{1}$ and $b$ is the number of steps for the last hit on $r_{2}$ back to
state $x_{0}$. Note that $a$ and $b$ all depend on $x_{0}$, but we suppress
the dependence for notational convenience. $N$ is an integer that we will pick momentarily.

Under $A(x_{0})$ we have
\begin{align*}
Z_{a+Nl+b}  &  =ze^{-\gamma_{1}-\cdots-\gamma_{a+Nl+b}}-\Delta\lambda
_{1}e^{-\gamma_{2}-\cdots-\gamma_{a+Nl+b}}-\Delta\lambda_{2}e^{-\gamma
_{3}-\cdots-\gamma_{a+Nl+b}}-\cdots-\Delta\lambda_{a+Nl+b}\\
&  \leq zu_{1}^{a+Nl+b}-\Delta u_{3}(u_{2}^{a+Nl}+u_{2}^{a+(N-1)l}%
+\cdots+u_{2}^{a})\\
&  \leq zu_{1}^{a+Nl+b}-\Delta u_{3}u_{2}^{a}\frac{1-u_{2}^{(N+1)l}}{1-u_{2}}%
\end{align*}
Now pick $N$ to be the smallest integer at least as large as
\[
\frac{\log((u_{1}^{a+b}Mc_{\Delta}\Delta(1-u_{2})+\Delta u_{3}u_{2}%
^{a+l})/(\Delta u_{3}u_{2}^{a}))}{l\log(1/u_{1})}%
\]
This implies that
\[
N<\frac{\log((z_{0}u_{1}^{a+b}(1-u_{2})+\Delta u_{3}u_{2}^{a+l})/(\Delta
u_{3}u_{2}^{a}))}{l\log(1/u_{1})}+1
\]
(note the definition of $C^{\prime}$ and $M$ above) and a simple verification
reveals that $Z_{a+Nl+b}\leq0$ on $A(x_{0})$. Hence if $A(x_{0})$ occurs, then
$T_{B}$ is hit before step $a+Nl+b$. Note that $N=O(\log c_{\Delta})$.

Now note that given $y_{0}=(x_{0},z_{0})\in C^{\prime}$, the probability that
$A(x_{0})$ happens is larger than $q^{a+Nl+b}$ for some $q>0$. If we divide
the steps of the chain into blocks of size $r+Nl$, where $r=\max
_{x\in\mathcal{S}}\{a(x)+b(x)\}$, then the number of blocks required for
$Z_{n}$ to hit 0 (and hence $T_{B\cup C}$ is achieved) is bounded by a
geometric r.v. with parameter $q^{r+Nl}$. Taking also into account the length
of the blocks, we have
\[
E_{y_{0}}^{\theta_{*}}T_{B\cup C}\leq\frac{1}{q^{r+Nl}}(r+Nl)=O(c_{\Delta}%
^{p}\log c_{\Delta})
\]
for some $p>0$.
\end{proof}

\bigskip

\begin{proof}
[Proof of Lemma \ref{lemma1}]Given an initial position $y_{0}=(x_{0},z_{0})\in
C^{\prime}$. It suffices to show that the path $A(x_{0})$ we have constructed
in the proof of Lemma \ref{lemma2} does not hit $T_{C}$ before $T_{B}$ i.e. it
does not hit $T_{C}$ for every step up through $a+Nl+b$. The conclusion of
Lemma \ref{lemma1} then follows by noting that $P_{y_{0}}^{\theta_{*}}%
(T_{B}>T_{C})\geq P_{y_{0}}^{\theta_{*}}(A(x_{0}))$. To prove $T_{C}$ is not
hit for every step, we show that $Z_{n}<c_{\Delta}\Delta u_{\theta_{*}}%
(X_{n})^{1/(2\theta_{*}-\rho_{\Delta})}u_{\theta_{*}-\rho_{\Delta}}%
(X_{n})^{1/(2\theta_{*}-\rho_{\Delta})}$ i.e. $Z_{n}\in B\cup C^{\prime}$, for
every $n=1,\ldots,a+Nl+b$ by induction. Suppose $Z_{n}<c_{\Delta}\Delta
u_{\theta_{*}}(X_{n})^{1/(2\theta_{*}-\rho_{\Delta})}u_{\theta_{*}%
-\rho_{\Delta}}(X_{n})^{1/(2\theta_{*}-\rho_{\Delta})}$, then
\begin{align*}
Z_{n+1}  &  =Z_{n}e^{-\gamma_{n+1}}-\Delta\lambda_{n+1}\\
&  \leq c_{\Delta}\Delta u_{\theta_{*}}(X_{n})^{1/(2\theta_{*}-\rho_{\Delta}%
)}u_{\theta_{*}-\rho_{\Delta}}(X_{n})^{1/(2\theta_{*}-\rho_{\Delta})}\cdot
u_{1}\frac{u_{\theta_{*}}(X_{n+1})^{1/\theta_{*}}}{u_{\theta_{*}}%
(X_{n})^{1/\theta_{*}}}\\
&  <c_{\Delta}\Delta u_{\theta_{*}}(X_{n+1})^{1/(2\theta_{*}-\rho_{\Delta}%
)}u_{\theta_{*}-\rho_{\Delta}}(X_{n+1})^{1/(2\theta_{*}-\rho_{\Delta})}%
\end{align*}
for small enough $\Delta$, where $u_{1}$ is defined in \eqref{prob1}, by
choosing the eigenvectors $u_{\theta_{*}-\delta}(x)$ that are continuous in
$\delta$ within a small neighborhood of 0 uniformly over all $x\in\mathcal{S}%
$. Hence we have proved our claim.
\end{proof}

\bigskip

\begin{proof}
[Proof of Lemma \ref{lemma3}]Suppose we start at $y_{0}=(x_{0},z_{0})\in C$.
Consider $\tilde{\gamma}_{n}=\sum_{i=\tau_{n-1}}^{\tau_{n}}\gamma_{i}$ where
$\tau_{n}=\inf\{i>\tau_{n-1}:X_{i}=x_{0}\}$. Inside region $C$ the random walk
$\tilde{S}_{n}=\sum_{j=1}^{n}\tilde{\gamma}_{j}$ has positive drift i.e.
$E\tilde{\gamma}_{n}>0$. For the process to hit $C^{\prime}\cup B$, it
suffices to have
\[
z_{0}e^{-\gamma_{1}-\cdots-\gamma_{n}}-\Delta\lambda_{1}e^{-\gamma_{2}%
-\cdots-\gamma_{n}}-\cdots-\Delta\lambda_{n}\leq mc_{\Delta}\Delta
\]
or equivalently
\[
\Delta(\lambda_{n}+mc_{\Delta}\Delta)e^{\gamma_{1}+\cdots+\gamma_{n}}%
+\Delta\lambda_{n-1}e^{\gamma_{1}+\cdots+\gamma_{n-1}}+\cdots+\Delta
\lambda_{1}e^{\gamma_{1}}\geq z_{0}
\]
This will be implied by the condition $\Delta(\lambda_{n}+mc_{\Delta}%
\Delta)e^{\gamma_{1}+\cdots+\gamma_{n}}\geq z_{0}$, which in turn can be
achieved if $S_{n}\geq\log(z_{0}/(mc_{\Delta}\Delta))$. Note that if we only
consider the steps when $X_{n}=x_{0}$, then the condition becomes $\tilde
{S}_{n}\geq\log(z_{0}/(mc_{\Delta}\Delta))$ where $\tilde{S}_{n}$ is now a
positively drifted random walk. This happens with probability one and the
expected time for this to happen is $O\left(  \log\left(  \frac{z_{0}%
}{mc_{\Delta}\Delta}\right)  \right)  $, which provides an upper bound for
$E_{y_{0}}^{\theta_{*}}[T_{C^{\prime}\cup B}]$.
\end{proof}

\bigskip

\begin{proof}
[Proof of Proposition \ref{termination}]Consider an initial position at
$y_{0}=(x_{0},z_{0})\in C$ (if $y_{0}=(x_{0},z_{0})\in C^{\prime}$ the same
analysis goes through resulting in a shorter mean running time). With
probability one $Y_{n}$ will enter $C^{\prime}\cup B$ by Lemma \ref{lemma3}.
If $T_{C^{\prime}}<T_{B}$ then by Lemma \ref{lemma1} the process hits $B$ with
probability that is bounded away from zero uniformly over $Y_{T_{C^{\prime}}}%
$, otherwise it goes back to $C$. Hence by geometric trial argument the
process will hit $B$ eventually. We obtain the first part of the proposition.

We now consider $E_{y_{0}}^{\theta_{*}}T_{B}$. Suppose first that
$y_{0}=(x_{0},z_{0})\in C^{\prime}$. Write
\[
E_{y_{0}}^{\theta_{*}}T_{B}=E_{y_{0}}^{\theta_{*}}[T_{B};T_{B}<T_{C}%
]+E_{y_{0}}^{\theta_{*}}[T_{B};T_{B}>T_{C}]
\]
Let $\bar{T}_{B}=T_{B}-T_{C}$ on the set $T_{B}>T_{C}$ i.e. $\bar{T}_{B}$ is
the residual time to hit $B$ once $T_{C}$ is first hit. We can write
\[
E_{y_{0}}^{\theta_{*}}T_{B}=E_{y_{0}}^{\theta_{*}}T_{B\cup C}+E_{y_{0}%
}^{\theta_{*}}[\bar{T}_{B};T_{B}>T_{C}]=E_{y_{0}}^{\theta_{*}}T_{B\cup
C}+E_{y_{0}}^{\theta_{*}}[E_{Y_{T_{C}}(y_{0})}^{\theta_{*}}\bar{T}_{B}%
;T_{B}>T_{C}]
\]
where $Y_{T_{C}}(y_{0})$ is the state at time $T_{C}$ (the $y_{0}$ as a
parameter emphasizes the dependence on the initial position $y_{0}$). We
further write
\begin{equation}
E_{y_{0}}^{\theta_{*}}T_{B}=E_{y_{0}}^{\theta_{*}}T_{B\cup C}+E_{y_{0}%
}^{\theta_{*}}[E_{Y_{T_{C}}(y_{0})}^{\theta_{*}}[\bar{T}_{B\cup C^{\prime}%
}+E_{Y_{\bar{T}_{C^{\prime}}}}^{\theta_{*}}[\bar{\bar{T}}_{B};\bar{T}_{B}%
>\bar{T}_{C^{\prime}}]];T_{B}>T_{C}] \label{equation1}%
\end{equation}
where $\bar{\bar{T}}_{B}=\bar{T}_{B}-\bar{T}_{C^{\prime}}$ on the set $\bar
{T}_{B}>\bar{T}_{C^{\prime}}$.

Let $f(y)=E_{y}^{\theta_{\ast}}T_{B}$. \eqref{equation1} leads to
\begin{align}
f(y_{0})  &  \leq E_{y_{0}}^{\theta_{\ast}}T_{B\cup C}+E_{y_{0}}^{\theta
_{\ast}}[E_{Y_{T_{C}}(y_{0})}^{\theta_{\ast}}\bar{T}_{B\cup C^{\prime}}%
;T_{B}>T_{C}]+\sup_{w\in C^{\prime}}f(w)P_{y_{0}}^{\theta_{\ast}}(T_{B}%
>T_{C})\nonumber\\
&  \leq E_{y_{0}}^{\theta_{\ast}}T_{B\cup C}+cE_{y_{0}}^{\theta_{\ast}}\left[
\log\left(  \frac{Y_{T_{C}}(y_{0})}{mc_{\Delta}\Delta}\right)  ;T_{B}%
>T_{C}\right]  +\sup_{w\in C^{\prime}}f(w)P_{y_{0}}^{\theta_{\ast}}%
(T_{B}>T_{C}) \label{equation2}%
\end{align}
where $c>0$ is a constant, using Lemma \ref{lemma3}. Now consider
\[
Y_{T_{C}}(y_{0})=z_{0}e^{-\gamma_{1}-\cdots-\gamma_{T_{C}}}-\Delta\lambda
_{1}e^{-\gamma_{2}-\cdots-\gamma_{T_{C}}}-\cdots-\Delta\lambda_{T_{C}}\leq
z_{0}e^{-\gamma_{1}-\cdots-\gamma_{T_{C}}}%
\]
and hence
\[
\log\left(  \frac{Y_{T_{C}}(y_{0})}{mc_{\Delta}\Delta}\right)  \leq\log
z_{0}-\gamma_{1}-\cdots-\gamma_{T_{C}}-\log(mc_{\Delta}\Delta)
\]
Now
\begin{align*}
E_{y_{0}}^{\theta_{\ast}}[-\gamma_{1}-\cdots-\gamma_{T_{C}};T_{B}>T_{C}]  &
\leq E_{y_{0}}^{\theta_{\ast}}[|\gamma_{1}|+\cdots+|\gamma_{T_{B\cup C}%
}|;T_{B}>T_{C}]\\
&  \leq E_{y_{0}}^{\theta_{\ast}}[|\gamma_{1}|+\cdots+|\gamma_{T_{B\cup C}%
}|]\\
&  \leq\tilde{c}E_{y_{0}}^{\theta_{\ast}}T_{B\cup C}%
\end{align*}
where $\tilde{c}=\sup_{x\in\mathcal{S}}E^{\theta_{\ast}}[|\gamma(X_{1},\xi
_{1})||X_{1}=x]<\infty$, by Wald's identity and Assumption 1 in Section 2.
This gives
\begin{equation}
E_{y_{0}}^{\theta_{\ast}}\left[  \log\left(  \frac{Y_{T_{C}}(y_{0}%
)}{mc_{\Delta}\Delta}\right)  ;T_{B}>T_{C}\right]  \leq\log z_{0}P_{y_{0}%
}^{\theta_{\ast}}(T_{B}>T_{C})+\tilde{c}E_{y_{0}}^{\theta_{\ast}}T_{B\cup
C}-\log(mc_{\Delta}\Delta)P_{y_{0}}^{\theta_{\ast}}(T_{B}>T_{C})
\label{equation3}%
\end{equation}
Putting \eqref{equation3} into \eqref{equation2} and using the fact that
$z_{0}\leq Mc_{\Delta}\Delta$ for $y_{0}=(x_{0},z_{0})\in C^{\prime}$ yields
\begin{align*}
f(y_{0})  &  \leq E_{y_{0}}^{\theta_{\ast}}T_{B\cup C}+cP_{y_{0}}%
^{\theta_{\ast}}(T_{B}>T_{C})\log(Mc_{\Delta}\Delta)+c\tilde{c}E_{y_{0}%
}^{\theta_{\ast}}T_{B\cup C}-c\log(mc_{\Delta}\Delta)P_{y_{0}}^{\theta_{\ast}%
}(T_{B}>T_{C}){}\\
&  {}+\sup_{w\in C^{\prime}}f(w)P_{y_{0}}^{\theta_{\ast}}(T_{B}>T_{C})
\end{align*}
Now taking supremum on both sides and using Lemma \ref{lemma2} and
\ref{lemma1}, we get
\[
\sup_{w\in C^{\prime}}f(w)\leq O\left(  c_{\Delta}^{p}\log c_{\Delta}%
+\log\left(  \frac{1}{c_{\Delta}\Delta}\right)  \frac{1}{c_{\Delta}^{p}%
}\right)
\]
Suppose we start with $y_{0}\in C$, then combining with Lemma \ref{lemma3}
concludes the proposition.
\end{proof}

\section{Extensions and Remarks\label{SectExtensions}}

As noted in equation (\ref{Dec1}) the perpetuity $D$ satisfies the
distributional fixed point equation
\[
D=_{d}B+AD^{\prime},
\]
where $A=\exp\left(  Y\right)  $, $D^{\prime}$ has the same distribution as
$D$, and $\left(  A,B\right)  $ is independent of $D^{\prime}$. We wish to
compare our development to that of Collamore et al (2011), which state their
results in the absence of Markov modulation, so to make the comparison more
transparent we will omit the Markov chain $\{X_{n}\}$ from our discussion
here. We have assumed that $B$ is non-negative but we believe that this is not
a strong assumption. We can typically reduce to the case of non-negative $B$.
Indeed, if the $B_{i}$'s can take negative values, we can let $\widetilde
{B}_{i}=\left\vert B_{i}\right\vert $ and define
\begin{equation}
\widetilde{D}=\widetilde{B}_{1}+\exp\left(  Y_{1}\right)  \widetilde{B}%
_{2}+.... \label{PerNeg}%
\end{equation}
Note that the tail of $\widetilde{D}$ is in great generality equivalent (up to
a constant) to that of $D$. Since $\{D>1/\Delta\}\subseteq\{\widetilde
{D}>1/\Delta\}$ we can use the likelihood ratio constructed to estimate the
tail of $\widetilde{D}$ but apply it to the event $I\left(  D>1/\Delta\right)
$. The same efficiency analysis applies automatically. So, throughout our
discussion we shall keep assuming that $B$ is non-negative.

Now, one could consider more general fixed point equations, for instance,
\begin{equation}
D=_{d}B+A\max\left(  D^{\prime},C\right)  \label{Dec1bis}%
\end{equation}
where $D^{\prime}$ has the same distribution as $D$, $\left(  A,B,C\right)  $
are independent of $D^{\prime}$. This equation is the focus of Collamore et al
(2011). Their assumptions are similar to ours, primarily that $\theta_{\ast
}>0$ satisfying $EA^{\theta_{\ast}}=1$ can be computed; that suitable moment
conditions are satisfied for $B$ and $C$, and that the associated
exponentially tilted distributions can be simulated. Their estimator for the
tail of $D$ is biased, but it enjoys asymptotic optimality properties parallel
to strong efficiency. In this sense this estimator is close in spirit to our
state-independent importance sampler. However, their construction is
completely different to ours, as we shall explain now.

Equation (\ref{Dec1bis}) characterizes the steady-state distribution (if it
exists) of the Markov chain, $\{V_{n}:n\geq0\}$, defined via $V_{0}=v_{0}$ and%
\begin{equation}
V_{n+1}=B_{n+1}+A_{n+1}\max\left(  V_{n},C_{n+1}\right)  , \label{MCC}%
\end{equation}
where $\{\left(  A_{n},B_{n},C_{n}\right)  :n\geq1\}$ is an i.i.d. sequence.
Collamore et al (2011) uses a regenerative ratio representation for the
steady-state distribution of $\{V_{n}\}$, assuming suitable minorization
conditions required for regeneration are in place. Clearly, there are
advantages to simulating $V_{n}$ (which is Markovian) as opposed to the
"backward" process, which corresponds to the discounted reward process (whose
limit is the perpetuity and it is not Markovian but requires keeping track of
$S_{n}=Y_{1}+...+Y_{n}$). Nevertheless, one of the important features of
importance sampling is that it can be applied to estimate conditional
expectations of sample path functions given the even of interest (in this case
$D>1/\Delta$). This is also why we wanted our algorithms to be developed under
the presence of Markovian modulation, without resorting to a decomposition
such as (\ref{Dec1}).

This feature, we believe, is quite attractive specially in some of the
applications behind our motivation to study discounted process, such as
insurance and finance. The problem with using the "forward" representation
(i.e. $\{V_{n}\}$ and the associated regenerative ratio) is that, while the
tail estimation of $D$ is preserved, it is difficult to use the associated
algorithms for estimation of conditional sample path expectations.

Finally, we point out that a similar coupling idea to the one used in the
construction of (\ref{PerNeg}) can be applied to reduce the analysis of
(\ref{Dec1bis}) to the case of standard perpetuities. In particular, define
$\widetilde{B}_{n+1}=B_{n+1}+A_{n+1}C_{n+1}$, and plug this definition into
(\ref{PerNeg}) to define $\widetilde{D}$. Then we have that $\widetilde{D}\geq
D$, where $D$ is the limit of the "backward" representation associated to
(\ref{Dec1bis}). Since $\widetilde{D}$ and $D$ are typically tail equivalent
(except for a constant), again we can proceed as indicated earlier. Because
$\{D>1/\Delta\}\subseteq\{\widetilde{D}>1/\Delta\}$ we can use the likelihood
ratio constructed to estimate the tail of $\widetilde{D}$ but apply it to the
event $I\left(  D>1/\Delta\right)  $. Again, bias is introduced because of the
infinite horizon nature of $D$, but the rare-event simulation problem has been
removed by the importance sampling strategy constructed based on
$\widetilde{D}$.

\section{Numerical Experiments\label{SectSDIS_Numerics}}

We run our algorithm for the ARCH(1) sequence in Example 1 with $\alpha
_{0}=1,\ 2$ and $\alpha_{1}=3/4,\ 4/5$. In this example there is no Markov
modulation. Using the transformation into $T_{n}$ as shown in Example 1, the
tail probability of the steady-state distribution of the ARCH(1) process with
target level $1/\Delta$ is equivalent to the tail probability of a perpetuity
with $\lambda(X_{i},\eta_{i})=\alpha_{0}$, $\gamma(X_{i},\xi_{i})=\log
\alpha_{1}+\log\chi_{i}^{2}$ where $\chi_{i}^{2}$ are i.i.d. chi-square
r.v.'s, and target level $\alpha_{1}/\Delta$. One can compute easily that
\[
Ee^{\theta\gamma(X_{i},\xi_{i})}=(2\alpha_{1})^{\theta}\frac{\Gamma
(\theta+1/2)}{\Gamma(1/2)}.
\]
and hence verify that Assumptions 1, 2 and 3 are satisfied. Moreover, the
conditions in Proposition \ref{PropLIa} are also satisfied with appropriate
selection of parameters (see the discussion below). Our choices of $\alpha
_{1}$ would correspond to $\theta_{\ast}$ with values $1.68$, $1.46$ and
$1.34$ respectively. This implies a tail of the steady-state ARCH(1) model
that has finite third moment but not the fourth, which frequently arises in
the financial context (see, for example, Mikosch and Starica (2000)).

We test the performance of both our state-independent and state-dependent
sampler proposed in Sections 3 and 4 by comparing with crude Monte Carlo. To
gauge the performance of our algorithms as $\Delta$ becomes small, we tune
$\Delta$ from $0.1$ to $0.00001$ to see the effect of the magnitude of
$\Delta$ to the output performance.

For crude Monte Carlo, we truncate the maximum number of steps to be 1000 (so
that the sequence does not iterate indefinitely; note that this would
certainly cause bias in the sample).

In the case of the state-independent importance sampler, we use $a=9/10$ and
$n_{\ast}=10\log(1/\Delta)$ (where $a$ is the proportion of the barrier that
upon touching would lead to the stop of importance sampling and $n_{\ast}$ is
the number of steps we continue to simulate after $T_{\Delta/a}$).

For the state-dependent sampler, we can verify that $b_{0}=1$, $b_{1}%
=\sup_{0\leq\zeta\leq\theta_{\ast}}(\psi^{\prime\prime}(\zeta)+(\psi^{\prime
2}\left(  \zeta\right)  ))/2$ and $b_{2}=\max\{\alpha_{0}^{2\theta_{\ast}%
},1\}$ satisfy the conditions in Proposition 2. To ensure that the Lyapunov
inequality holds for small $\Delta$ one can choose $B_{1}$ and $B_{2}$ in
Proposition 2 to satisfy $B_{2}\geq1$ and
\[
\mu-2B_{1}\theta_{\ast}-\frac{b_{0}b_{2}}{B_{2}^{\theta_{\ast}}}>0
\]
In particular we can choose $B_{1}=0.45\mu/(2\theta_{\ast})$ and $B_{2}%
=\max\{(b_{0}b_{2}/(0.45\mu))^{1/\theta_{\ast}},1\}$.

For each set of input parameters (i.e. $\alpha_{0}$, $\alpha_{1}$ and $\Delta
$) we simulate using crude Monte Carlo, the state-independent sampler, and the
state-dependent sampler. For each method we fix the running time to be five
minutes for comparison. In the following tables we show the estimate,
empirical coefficient of variation (standard deviation divided by the
estimate), and 95\% confidence interval for our simulation. Tables are
deferred to the appendix below.

Next we also run the algorithm for a Markov modulated perpetuity. The
modulating Markov chain lies in state space $\{1,2\}$ and has transition
matrix
\[
K=\left[
\begin{array}
[c]{ll}%
\frac{1}{2} & \frac{1}{2}\\
1 & 0
\end{array}
\right]
\]
We use $\lambda(1,\eta_{1})=\lambda(1)=1$ and $\lambda(2,\eta_{1}%
)=\lambda(2)=2$ a.s.. Also, $\gamma(1,\xi_{1})=\log(2/3)+\log\chi_{i}^{2}$ and
$\gamma(2,\xi_{1})=\log(3/4)+\log\chi_{i}^{2}$, where again $\chi_{i}^{2}$ are
i.i.d. chi-square random variables.

Again we experiment using crude Monte Carlo and both state-independent and
dependent importance samplers. Similar to the ARCH(1) setup, \ a simple
calculation reveals that $e^{\chi(1,\theta)}=(2\times(2/3))^{2\theta}%
\Gamma(\theta+1/2)/\Gamma(1/2)$ and $e^{\chi(2,\theta)}=(2\times
(3/4))^{2\theta}\Gamma(\theta+1/2)/\Gamma(1/2)$. Moreover, $e^{\psi(\theta
)}=(e^{\chi(1,\theta)}+\sqrt{e^{2\chi(1,\theta)}+8e^{\chi(1,\theta
)+\chi(2,\theta)}})/4$, so $\theta_{\ast}=1.60$. We take $u_{\theta}%
(1)=(\sqrt{e^{2\chi(1,\theta)}+8e^{\chi(1,\theta)+\chi(2,\theta)}}%
+e^{\chi(1,\theta)})/(4e^{\chi(1,\theta)})$ and $u_{\theta}(2)=1$.

The computational effort needed for this Markov-modulated problem appears to
be substantially heavier than the case of ARCH model, and hence we perform a
longer and more extensive simulation study. We tune $\Delta$ from $0.1$ to
$0.002$, and for each scenario we run the simulation for one hour for each
method. For crude Monte Carlo, we use $100,000$ as our step truncation. For
state-independent importance sampler we use $a=9/10$ and $n_{\ast}%
=1000\log(1/\Delta)$. For state-dependent importance sampler, the setup according to Proposition 1 is as follows. First, for every $\theta\in(0,\theta_*)$, we use a normalization of $u_\theta(\cdot)$ such that $\min_{x\in\mathcal S}u_\theta(x)=1$. This choice of $u_\theta(\cdot)$ satisfies continuity in $\theta$ in a small neighborhood $(\theta_*-\rho_\Delta,\theta_*)$ (regarding $u_\theta(\cdot)$ as a vector-valued function). This is because the largest eigenvalue of the matrix $Q_\theta$ is isolated; a small perturbation of $\theta$ leads to a small change in the angle of the associated eigenspace, and hence the eigenvector taken as the direction of the eigenspace is continuous in $\theta$. Consequently, the minimum taken over the components of the eigenvector is also continuous in $\theta$, and so is the eigenvector in the particular normalization that we use. As a result, we can choose $b_{0}=1$. For other parameters, we then use $b_{1}=\sup_{\zeta\in(0,\theta_{\ast})}(\psi^{\prime\prime}%
(\zeta)+\psi^{\prime2}\left(  \zeta\right)  )/2$, and $b_{2}=\max\{\sup
_{x\in\mathcal{S}}\lambda(x)^{2\theta_{\ast}},1\}\sup_{x\in\mathcal{S}}%
E_{x}e^{\chi(X_{1},\theta_{\ast})}$. Then taking $B_{1}=0.45\mu/(2\theta
_{\ast})$ and $B_{2}=\max\{(b_{0}b_{2}/(0.45\mu))^{1/\theta_{\ast}},1\}$ will
satisfy the Lyapunov inequality for small enough $\Delta$. The numerical
outputs are shown in the appendix below.

For the ARCH model, it is notable from the coefficient of variation and
confidence interval that both state-independent and state-dependent samplers
perform better than crude Monte Carlo starting from $\Delta=0.001$. Crude
Monte Carlo has much larger coefficient of variation when $\Delta$ is 0.0005,
and it merely fails (i.e.\ does not generate any positive sample) when
$\Delta$ is 0.00001. On the other hand, the coefficient of variation for
state-independent sampler remains at around 1 to 2 and that for
state-dependent sampler remains under 50 for all the cases we considered. The
state-independent sampler appears to perform better than state-dependent
sampler for our range of $\Delta$, although one should keep in mind there is
bias issue in that algorithm. Similar results hold for the Markov-modulated
perpetuity, where crude Monte Carlo fails completely when $\Delta$ is 0.005 or
larger while the importance samplers still perform reasonably well.

%On a separate note, we find that by choosing a small $c_\Delta$ (can be so small that even the Lyapunov inequality does not hold e.g. $c_\Delta=1$) in the state-dependent algorithm can lead to a significantly smaller running time in some cases, especially when $\theta_*$ is large so that the effect of $c_\Delta^{2\theta_*}$ to these quantities is magnified. Nevertheless, we are contented that by trading a larger $c_\Delta$, and hence running time, for efficiency we end up with an algorithm that is guaranteed to be asymptotically optimal and still perform at an acceptable speed.

\subsection{Appendix: Numerical Output}

ARCH model, parameter values: $\alpha_{0}=1,\ \alpha_{1}=3/4$%

\begin{tabular}
[c]{l|ccc}%
\multicolumn{4}{c}{Crude Monte Carlo}\\
\multicolumn{4}{c}{}\\
& Estimate & C.V. & 95\% C.I.\\\hline
$\Delta=0.1$ & $6.65\times10^{-2}$ & 3.75 & $[6.43\times10^{-2},6.86\times
10^{-2}]$\\
$\Delta=0.05$ & $2.89\times10^{-2}$ & 5.80 & $[2.74\times10^{-2},
3.04\times10^{-2}]$\\
$\Delta=0.001$ & $1.11\times10^{-4}$ & 95.05 & $[1.37\times10^{-5}%
,2.08\times10^{-4} ]$\\
$\Delta=0.0005$ & $2.11\times10^{-5}$ & 217.9 & $[-2.02\times10^{-5}%
,6.24\times10^{-5}]$\\
$\Delta=0.00001$ & 0 & N/A & N/A
\end{tabular}

\begin{tabular}
[c]{l|ccc}%
\multicolumn{4}{c}{State-Independent Sampler}\\
\multicolumn{4}{c}{}\\
& Estimate & C.V. & 95\% C.I.\\\hline
$\Delta=0.1$ & $6.84\times10^{-2}$ & 1.79 & $[6.82\times10^{-2},6.86\times
10^{-2}]$\\
$\Delta=0.05$ & $2.84\times10^{-2}$ & 1.76 & $[2.83\times10^{-2}%
,2.85\times10^{-2}]$\\
$\Delta=0.001$ & $1.10\times10^{-4}$ & 1.75 & $[1.09\times10^{-4}%
,1.10\times10^{-4}]$\\
$\Delta=0.0005$ & $4.01\times10^{-5}$ & 1.79 & $[3.99\times10^{-5}%
,4.02\times10^{-5}]$\\
$\Delta=0.00001$ & $1.34\times10^{-7}$ & 1.72 & $[1.33\times10^{-7}%
,1.35\times10^{-7}]$%
\end{tabular}

\begin{tabular}
[c]{l|ccc}%
\multicolumn{4}{c}{State-Dependent Sampler}\\
\multicolumn{4}{c}{}\\
& Estimate & C.V. & 95\% C.I.\\\hline
$\Delta=0.1$ & $6.80\times10^{-2}$ & 3.68 & $[6.63\times10^{-2},6.96\times
10^{-2}]$\\
$\Delta=0.05$ & $2.82\times10^{-2}$ & 5.81 & $[2.67\times10^{-2}%
,2.96\times10^{-2}]$\\
$\Delta=0.001$ & $1.45\times10^{-4}$ & 24.77 & $[7.70\times10^{-5}%
,2.14\times10^{-4}]$\\
$\Delta=0.0005$ & $4.22\times10^{-5}$ & 29.93 & $[1.49\times10^{-5}%
,6.95\times10^{-5}]$\\
$\Delta=0.00001$ & $2.48\times10^{-7}$ & 37.71 & $[-5.43\times10^{-8}%
,5.49\times10^{-7}]$%
\end{tabular}

ARCH model, parameter values: $\alpha_{0}=2,\ \alpha_{1}=3/4$%

\begin{tabular}
[c]{l|ccc}%
\multicolumn{4}{c}{Crude Monte Carlo}\\
\multicolumn{4}{c}{}\\
& Estimate & C.V. & 95\% C.I.\\\hline
$\Delta=0.1$ & $1.51\times10^{-1}$ & 2.37 & $[1.48\times10^{-1},1.54\times
10^{-1}]$\\
$\Delta=0.05$ & $6.64\times10^{-2}$ & 3.75 & $[6.40\times10^{-2}%
,6.88\times10^{-2}]$\\
$\Delta=0.001$ & $2.61\times10^{-4}$ & 61.92 & $[1.13\times10^{-4}%
,4.08\times10^{-4}]$\\
$\Delta=0.0005$ & $8.19\times10^{-5}$ & 110.5 & $[1.64\times10^{-6}%
,1.62\times10^{-4}]$\\
$\Delta=0.00001$ & 0 & N/A & N/A
\end{tabular}

\begin{tabular}
[c]{l|ccc}%
\multicolumn{4}{c}{State-Independent Sampler}\\
\multicolumn{4}{c}{}\\
& Estimate & C.V. & 95\% C.I.\\\hline
$\Delta=0.1$ & $1.50\times10^{-1}$ & 1.92 & $[1.495\times10^{-1}%
,1.503\times10^{-1}]$\\
$\Delta=0.05$ & $6.85\times10^{-2}$ & 2.48 & $[6.83\times10^{-2}%
,6.88\times10^{-2}]$\\
$\Delta=0.001$ & $3.00\times10^{-4}$ & 1.87 & $[2.99\times10^{-4}%
,3.02\times10^{-4}]$\\
$\Delta=0.0005$ & $1.09\times10^{-4}$ & 1.69 & $[1.09\times10^{-4}%
,1.10\times10^{-4}]$\\
$\Delta=0.00001$ & $3.69\times10^{-7}$ & 1.69 & $[3.67\times10^{-7}%
,3.71\times10^{-7}]$%
\end{tabular}

\begin{tabular}
[c]{l|ccc}%
\multicolumn{4}{c}{State-Dependent Sampler}\\
\multicolumn{4}{c}{}\\
& Estimate & C.V. & 95\% C.I.\\\hline
$\Delta=0.1$ & $1.50\times10^{-1}$ & 2.38 & $[1.46\times10^{-1},1.54\times
10^{-1}]$\\
$\Delta=0.05$ & $6.92\times10^{-2}$ & 3.66 & $[6.57\times10^{-2}%
,7.27\times10^{-2}]$\\
$\Delta=0.001$ & $5.54\times10^{-4}$ & 42.46 & $[-2.14\times10^{-4}%
,1.32\times10^{-3}]$\\
$\Delta=0.0005$ & $1.61\times10^{-5}$ & 42.72 & $[-8.85\times10^{-6}%
,4.11\times10^{-5}]$\\
$\Delta=0.00001$ & $9.13\times10^{-9}$ & 34.65 & $[-8.78\times10^{-9}%
,2.70\times10^{-8}]$%
\end{tabular}

ARCH model, parameter values: $\alpha_{0}=1,\ \alpha_{1}=4/5$%

\begin{tabular}
[c]{l|ccc}%
\multicolumn{4}{c}{Crude Monte Carlo}\\
\multicolumn{4}{c}{}\\
& Estimate & C.V. & 95\% C.I.\\\hline
$\Delta=0.1$ & $7.79\times10^{-2}$ & 3.44 & $[7.45\times10^{-2},
8.14\times10^{-2} ]$\\
$\Delta=0.05$ & $3.53\times10^{-2}$ & 5.23 & $[3.29\times10^{-2},
3.76\times10^{-2}]$\\
$\Delta=0.001$ & $2.27\times10^{-4}$ & 66.39 & $[2.80\times10^{-5},
4.26\times10^{-4}]$\\
$\Delta=0.0005$ & $9.13\times10^{-5}$ & 104.6 & $[-3.52\times10^{-5},
2.18\times10^{-4}]$\\
$\Delta=0.00001$ & 0 & N/A & N/A
\end{tabular}

\begin{tabular}
[c]{l|ccc}%
\multicolumn{4}{c}{State-Independent Sampler}\\
\multicolumn{4}{c}{}\\
& Estimate & C.V. & 95\% C.I.\\\hline
$\Delta=0.1$ & $7.78\times10^{-2}$ & 1.72 & $[7.75\times10^{-2},
7.81\times10^{-2}]$\\
$\Delta=0.05$ & $3.43\times10^{-2}$ & 1.56 & $[3.41\times10^{-2},
3.44\times10^{-2}]$\\
$\Delta=0.001$ & $2.02\times10^{-4}$ & 1.57 & $[2.01\times10^{-4},
2.03\times10^{-4}]$\\
$\Delta=0.0005$ & $8.00\times10^{-5}$ & 1.53 & $[7.96\times10^{-5},
8.05\times10^{-5}]$\\
$\Delta=0.00001$ & $4.21\times10^{-7}$ & 1.55 & $[4.18\times10^{-7},
4.24\times10^{-7}]$%
\end{tabular}

\begin{tabular}
[c]{l|ccc}%
\multicolumn{4}{c}{State-Dependent Sampler}\\
\multicolumn{4}{c}{}\\
& Estimate & C.V. & 95\% C.I.\\\hline
$\Delta=0.1$ & $7.84\times10^{-2}$ & 3.39 & $[7.62\times10^{-2},
8.07\times10^{-2}]$\\
$\Delta=0.05$ & $3.47\times10^{-2}$ & 5.21 & $[3.27\times10^{-2},
3.67\times10^{-2}]$\\
$\Delta=0.001$ & $1.31\times10^{-4}$ & 27.97 & $[4.61\times10^{-5},
2.16\times10^{-4}]$\\
$\Delta=0.0005$ & $6.52\times10^{-5}$ & 23.15 & $[2.71\times10^{-5},
1.03\times10^{-4}]$\\
$\Delta=0.00001$ & $8.09\times10^{-7}$ & 26.51 & $[ -6.15\times10^{-9},
1.62\times10^{-6}]$%
\end{tabular}

ARCH model, parameter values: $\alpha_{0}=2,\ \alpha_{1}=4/5$%

\begin{tabular}
[c]{l|ccc}%
\multicolumn{4}{c}{Crude Monte Carlo}\\
\multicolumn{4}{c}{}\\
& Estimate & C.V. & 95\% C.I.\\\hline
$\Delta=0.1$ & $1.59\times10^{-1}$ & 2.30 & $[1.55\times10^{-1},
1.64\times10^{-1} ]$\\
$\Delta=0.05$ & $7.96\times10^{-2}$ & 3.40 & $[7.62\times10^{-2},
8.30\times10^{-2}]$\\
$\Delta=0.001$ & $4.49\times10^{-4}$ & 47.20 & $[1.71\times10^{-4},
7.27\times10^{-4}]$\\
$\Delta=0.0005$ & $2.69\times10^{-4}$ & 60.92 & $[5.39\times10^{-5},
4.85\times10^{-4}]$\\
$\Delta=0.00001$ & 0 & N/A & N/A
\end{tabular}

\begin{tabular}
[c]{l|ccc}%
\multicolumn{4}{c}{State-Independent Sampler}\\
\multicolumn{4}{c}{}\\
& Estimate & C.V. & 95\% C.I.\\\hline
$\Delta=0.1$ & $1.62\times10^{-1}$ & 1.67 & $[1.61\times10^{-1},
1.62\times10^{-1}]$\\
$\Delta=0.05$ & $7.74\times10^{-2}$ & 1.57 & $[7.71\times10^{-2},
7.77\times10^{-2}]$\\
$\Delta=0.001$ & $5.12\times10^{-4}$ & 1.57 & $[5.10\times10^{-4},
5.15\times10^{-4}]$\\
$\Delta=0.0005$ & $2.03\times10^{-4}$ & 1.74 & $[2.01\times10^{-4},
2.04\times10^{-4}]$\\
$\Delta=0.00001$ & $1.07\times10^{-6}$ & 1.59 & $[ 1.06\times10^{-6},
1.08\times10^{-6}]$%
\end{tabular}

\begin{tabular}
[c]{l|ccc}%
\multicolumn{4}{c}{State-Dependent Sampler}\\
\multicolumn{4}{c}{}\\
& Estimate & C.V. & 95\% C.I.\\\hline
$\Delta=0.1$ & $1.61\times10^{-1}$ & 2.28 & $[1.56\times10^{-1},
1.65\times10^{-1}]$\\
$\Delta=0.05$ & $7.77\times10^{-2}$ & 3.44 & $[7.33\times10^{-2},
8.22\times10^{-2}]$\\
$\Delta=0.001$ & $4.32\times10^{-4}$ & 44.24 & $[-2.64\times10^{-4},
1.13\times10^{-3}]$\\
$\Delta=0.0005$ & $1.39\times10^{-4}$ & 23.65 & $[ 6.72\times10^{-6},
2.71\times10^{-4}]$\\
$\Delta=0.00001$ & $5.93\times10^{-7}$ & 32.41 & $[-5.70\times10^{-7},
1.76\times10^{-6}]$%
\end{tabular}

Markov-modulated perpetuity%

\begin{tabular}
[c]{l|ccc}%
\multicolumn{4}{c}{Crude Monte Carlo}\\
\multicolumn{4}{c}{}\\
& Estimate & C.V. & 95\% C.I.\\\hline
$\Delta=0.1$ & $7.23\times10^{-2}$ & 3.58 & $[5.64\times10^{-2},
8.82\times10^{-2}]$\\
$\Delta=0.05$ & $2.61\times10^{-2}$ & 6.12 & $[1.60\times10^{-2},
3.62\times10^{-2}]$\\
$\Delta=0.02$ & $3.09\times10^{-3}$ & 17.98 & $[-4.07\times10^{-4},
6.58\times10^{-3}]$\\
$\Delta=0.005$ & 0 & N/A & N/A\\
$\Delta=0.002$ & 0 & N/A & N/A
\end{tabular}

\begin{tabular}
[c]{l|ccc}%
\multicolumn{4}{c}{State-Independent Sampler}\\
\multicolumn{4}{c}{}\\
& Estimate & C.V. & 95\% C.I.\\\hline
$\Delta=0.1$ & $5.82\times10^{-2}$ & 45.89 & $[5.55\times10^{-2},
6.10\times10^{-2}]$\\
$\Delta=0.05$ & $2.08\times10^{-2}$ & 10.45 & $[2.04\times10^{-2},
2.11\times10^{-2}]$\\
$\Delta=0.02$ & $5.24\times10^{-3}$ & 14.67 & $[5.13\times10^{-3},
5.34\times10^{-3}]$\\
$\Delta=0.005$ & $5.92\times10^{-4}$ & 11.25 & $[5.73\times10^{-4},
6.10\times10^{-4}]$\\
$\Delta=0.002$ & $1.37\times10^{-4}$ & 9.25 & $[1.34\times10^{-4},
1.40\times10^{-4}]$%
\end{tabular}

\begin{tabular}
[c]{l|ccc}%
\multicolumn{4}{c}{State-Dependent Sampler}\\
\multicolumn{4}{c}{}\\
& Estimate & C.V. & 95\% C.I.\\\hline
$\Delta=0.1$ & $5.73\times10^{-2}$ & 4.05 & $[5.35\times10^{-2},6.10\times
10^{-2}]$\\
$\Delta=0.05$ & $2.23\times10^{-2}$ & 6.62 & $[1.88\times10^{-2}%
,2.58\times10^{-2}]$\\
$\Delta=0.02$ & $3.51\times10^{-3}$ & 16.83 & $[1.74\times10^{-3}%
,5.27\times10^{-3}]$\\
$\Delta=0.005$ & $4.40\times10^{-4}$ & 47.68 & $[-4.23\times10^{-4}%
,1.30\times10^{-3}]$\\
$\Delta=0.002$ & $2.35\times10^{-5}$ & 44.40 & $[-2.26\times10^{-5}%
,6.96\times10^{-5}]$%
\end{tabular}

\section*{References}

\begin{enumerate}
\item Adler, R., Blanchet, J., and Liu, J. C. (2010) Fast Simulation of
Gaussian Random Fields with High Excursions. \textit{Preprint}.

\item Asmussen (2000), \emph{Ruin Probabilities}, World Scientific.

\item Asmussen (2003), \emph{Applied Probability and Queues}, Second Edition, Springer.

\item Asmussen, S. and Glynn, P. (2008), \emph{Stochastic Simulation:
Algorithms and Analysis}, Springer-Verlag.

\item Asmussen, S. and Nielsen, H. M. (1995), Ruin probabilities via local
adjustment coefficients, \emph{J. Appl. Prob.}, \textbf{32}, 736-755.

\item Benoite de Saporta (2005). Tail of the stationary solution of the
stochastic equation $Y(n+1)=a(n)Y(n)+b(n)$ with Markovian coefficients,
\emph{Stoc. Proc. Appl.} \textbf{115}(12), 1954--1978.

\item Blanchet, J. and Glynn, P. (2008), Efficient rare-event simulation for
the maximum of heavy-tailed random walks, \emph{Ann. Appl. Prob.},
\textbf{18}(4), 1351-1378.

\item Blanchet, J., Glynn, P., and Liu, J. C. (2007), Fluid heuristics,
Lyapunov bounds and efficient importance sampling for a heavy-tailed G/G/1
queue, \emph{QUESTA}, \textbf{57 }(2-3), 99--113.

\item Blanchet, J., Leder, K., and Glynn, P. (2009), Lyapunov functions and
subsolutions for rare event simulation, \emph{Preprint}.

\item Blanchet, J. and Sigman, K. (2011), Perfect Sampling of Perpetuities,
\emph{Journal of Applied Probability}, Special Vol. \textbf{48A}, 165-183.

\item Blanchet, J. and Zwart, B. (2007), Importance sampling of compounding
processes, \emph{WSC '07: Proceedings of the 39th conference on Winter
simulation}, 372-379.

\item Collamore, J. F. (2002), Importance sampling techniques for the
multidimensional ruin problem for general Markov additive sequences of random
vectors, \emph{Ann. Appl. Prob.}, \textbf{12}(1), 382-421.

\item Collamore, J. F. (2009), Random recurrence equations and ruin in a
Markov-dependent stochastic economic environment, \emph{Ann. Appl. Prob.},
\textbf{19}(4), 1404-1458.

\item Collamore, J. F., Diao, G., and Vidyashankar A. N. (2011), Rare event
simulation for processes generated via stochastic fixed point equations.
\emph{Preprint.}

\item Diaconis, P. and Freedman, D. (1999), Iterated random functions,
\emph{SIAM Review}, \textbf{41}(1), 45-76.

\item Dufresne, D. (1990), The distribution of a perpetuity, with applications
to risk theory and pension funding, \emph{Scandinavian Actuarial Journal},
\textbf{1}(2), 39-79.

\item Dupuis, P. and Wang, H. (2004), Importance sampling, large deviations,
and differential games, \emph{Stoc. and Stoc. Reports}, \textbf{76}, 481-508.

\item Dupuis, P. and Wang, H. (2007), Subsolutions of an Isaacs equation and
efficient schemes for importance sampling: Convergence analysis, \emph{Math.
Oper. Res.} \textbf{32}, 723--757.

\item Embrechts, P., Kl\"{u}ppelberg, C., and Mikosch, T. (1997),
\emph{Modelling Extremal Events for Insurance and Finance}, Springer.

\item Enriquez, N., Sabot, C., Zindy, O. (2009). A probabilistic
representation of constants in Kestens renewal theorem. \emph{Prob. Theory and
Related Fields}, \textbf{144}, 581--613.

\item Glasserman, P. and Kou, S. (1995), Analysis of an importance sampling
estimator for tandem queues, \emph{ACM Transactions on Modeling and Computer},
\textbf{5}(1), 22-42.

\item Goldie, C. M. (1991), Implicit renewal theory and tails of solutions of
random equations, \emph{Ann. Appl. Prob.}, \textbf{1}, 126-166.

\item Kesten, H. (1973), Random difference equations and Renewal theory for
products of random matrices , \emph{Acta Mathematica}, \textbf{131}(1), 207-248.

\item Liu, J. (2001), \emph{Monte Carlo Strategies in Scientific Computing}, Springer.

\item Mikosch, T. and Starica, C. (2000), Limit theory for the sample
autocorrelations and extremes of a Garch(1,1) process, \emph{Ann. Stat.},
\textbf{28}, 1427-1451.

\item Nyrhinen, H. (2001), Finite and infinite time ruin probabilities in a
stochastic economic environment, \emph{Stoc. Proc. and Appl.}, \textbf{92}, 265-285.

\item Pollack, M. and Siegmund, D. (1985), A diffusion process and its
applications to detecting a change of sign of Brownian motion,
\emph{Biometrika}, \textbf{72}(2), 267-280.
\end{enumerate}

\end{document}